\documentclass[11pt]{amsart}
\usepackage{amssymb}
\usepackage{graphicx}
\usepackage{xcolor} 
\usepackage{tensor}
\usepackage{fullpage} 
\usepackage{tikz}
\usepackage{amsmath}
\usepackage{scalerel}

\usepackage{amsthm}
\usepackage{verbatim}
\usepackage{hyperref}
\usepackage{array} 
\usepackage{enumitem}

\setlist[enumerate]{leftmargin=1.8em}
\setlist[itemize]{leftmargin=1.8em}
\usepackage{seqsplit,mathtools}

\definecolor{green}{rgb}{0,0.8,0} 

\newtheorem{theorem}{Theorem}[section]
\newtheorem{corollary}[theorem]{Corollary}
\newtheorem{lemma}[theorem]{Lemma}

\newtheorem{proposition}[theorem]{Proposition}

\newtheorem{innercustomthm}{Theorem}
\newenvironment{customthm}[1]
{\renewcommand\theinnercustomthm{#1}\innercustomthm}
{\endinnercustomthm}
\theoremstyle{definition}

\theoremstyle{remark}
\newtheorem{remark}[theorem]{Remark}

\numberwithin{equation}{section}
\newcommand{\nrm}[1]{\Vert#1\Vert}

\newcommand{\tld}[1]{\widetilde{#1}}

\newcommand{\nnrm}[1]{{\vert\kern-0.25ex\vert\kern-0.25ex\vert #1 
		\vert\kern-0.25ex\vert\kern-0.25ex\vert}}

\newcommand{\dist}{\mathrm{dist}\,}

\newcommand{\supp}{{\mathrm{supp}}\,}

\newcommand{\lap}{\Delta}

\newcommand{\rd}{\partial}
\newcommand{\nb}{\nabla}

\newcommand{\alp}{\alpha}

\newcommand{\Gmm}{\Gamma}
\newcommand{\dlt}{\delta}

\newcommand{\kpp}{\kappa}
\newcommand{\lmb}{\lambda}

\newcommand{\tht}{\theta}

\newcommand{\omg}{\omega}
\newcommand{\Omg}{\Omega}

\newcommand{\varep}{\varepsilon}

\newcommand{\bfe}{{\bf e}}

\newcommand{\bfp}{{\bf p}}

\newcommand{\bfF}{{\bf F}}
\newcommand{\bfG}{{\bf G}}

\newcommand{\bfK}{{\bf K}}

\newcommand{\bbR}{\mathbb R}

\newcommand{\calR}{\mathcal R}

\begin{document}

	\title{Stability for multiple Lamb dipoles}

	\author{Ken Abe}
	\address{Department of Mathematics, Graduate School of Science, Osaka Metropolitan University, 3-3-138 Sugimoto, Sumiyoshi-ku Osaka, 558-8585, Japan.}
	\email{kabe@omu.ac.jp}
	
	\author{In-Jee Jeong} 
	\address{School of Mathematics, Korea Institute for Advanced Study, 85 Hoegi-ro, Seoul 02455, Republic of Korea.}
	\email{ijeong@kias.re.kr}
	
	\author{Yao Yao}
	\address{Department of Mathematics, National University of Singapore, Block S17, 10 Lower Kent Ridge Road, Singapore, 119076, Singapore.}
	\email{yaoyao@nus.edu.sg}

	\date{\today}

	\renewcommand{\thefootnote}{\fnsymbol{footnote}}
	\footnotetext{\emph{2020 AMS Mathematics Subject Classification:} 76B47, 35Q35, 35B40}
	\footnotetext{\emph{Key words: vortex stability, Lamb dipole, variational principle, Lagrangian bootstrapping, multi-soliton stability} }
	\renewcommand{\thefootnote}{\arabic{footnote}}

	\begin{abstract} 
		In the class of nonnegative vorticities on the half-plane, we establish the Lyapunov stability of finite sums of Lamb dipoles under the initial assumptions that the dipoles are sufficiently separated and that the faster dipoles are positioned to the right of the slower ones. Our approach combines sharp energy estimates near the Lamb dipoles with a Lagrangian bootstrapping scheme, enabling us to quantify the exchanges of circulation, enstrophy, impulse, and energy between various parts of the solution. The strategy of the proof is robust, and we present several potential extensions of the result. 
	\end{abstract}
	
	\maketitle 
	
	
	\date{\today}         
	
	\section{Introduction}\label{sec:intro}
	
	In this paper, we are concerned with long-time dynamics of \textbf{nonnegative vorticities} in the upper half-plane $\bbR^2_+ = \left\{ x = (x_1,x_2) : x_{2} \ge 0 \right\}$, evolving by the incompressible Euler equation: \begin{equation}\label{eq:2D-Euler}
		\left\{
		\begin{aligned}
			\rd_t \omg + u \cdot \nb \omg = 0, & \\
			u = \nb^{\perp} \lap^{-1} \omg ,&
		\end{aligned}
		\right.
	\end{equation}  where $\nb^\perp = (-\rd_{x_2}, \rd_{x_1})^\top$ and $\lap^{-1}$ is the Dirichlet Laplacian inverse in $\bbR^2_+$. We shall always assume that the initial data $\omg_{0} = \omg(t=0)$ belongs to $L^{1} \cap L^{\infty}(\bbR^2_+)$, so that there is a unique global-in-time solution by Yudovich theory \cite{Y1}. 
	
	Although global well-posedness is well-known, the long-time dynamics of \eqref{eq:2D-Euler} is a highly challenging problem: while the vorticity seems to generally concentrate on just a few points as $t \to \infty$  (\cite{McW,BoVe,DE,CMV}), at the same time, it is known that complicated quasi-periodic motion can persist (\cite{Kha,BHM,CrF,EPT,HaRo,GSIP}). Regarding this problem, the setting of nonnegative vortices in $\bbR^2_+$ could offer a simplification, thanks to the following monotone behavior (\cite{ISG99,Iftimie1}):  \begin{equation}\label{eq:x1-mom-mon}
		\begin{split}
			c_{0}(\omg_{0}) \le \frac{d}{dt} \int_{\bbR^2_+} x_{1}\omg(t,x)\,dx =  \int_{\bbR^2_+} u_{1}(t,x)\omg(t,x)\,dx \le C_{0}(\omg_{0}) 
		\end{split}
	\end{equation} where $c_{0}(\omg_{0}) \le C_{0}(\omg_{0})$ are strictly positive constants (unless $\omg_{0}$ is identically zero) depending only on $\omg_{0}$. That is, on the vorticity support, the ``fluid particles'' are drifting to the right in average with $O(1)$ speed:  then it is conceivable that if the vorticity consists of a few separated vortices, the faster ones will escape to spatial infinity earlier than the slower ones, making the interaction smaller as $t\to\infty$. In the current work, we present several definite results in this direction, namely forward-in-time  stability and spatial separation of multi-vortex solutions. 
	
	Let us recall a few important conserved quantities for \eqref{eq:2D-Euler}. Any $L^{p}$ norm ($p \in [1,\infty]$), as well as the distribution function of $\omg(t,\cdot)$ is preserved in time, since the velocity is divergence-free. In relation to the Lamb dipoles, the \emph{enstrophy} $$K[\omg(t,\cdot)] := \nrm{\omg(t,\cdot)}_{L^{2}(\bbR^2_+)}^{2},$$ which is simply the square of the $L^{2}$ norm, plays a particularly important role. Furthermore, we introduce the \emph{impulse}  (the $L^{1}$ norm with weight $x_{2}$)  $$\mu[\omg(t,\cdot)] := \nrm{\omg(t,\cdot)}_{L^1_*} := 
	\int_{\bbR^2_+} x_{2}|\omg(t,x)| dx,$$ and the \emph{kinetic energy} \begin{equation*}
		\begin{split}
			E[\omg(t,\cdot)] := -\frac12 \int_{\bbR^2_+} \omg(t,x) (\lap^{-1}\omg)(t,x) dx = \frac12 \int_{\bbR^2_+} |u(t,x)|^{2} dx . 
		\end{split}
	\end{equation*} The kinetic energy is conserved by \eqref{eq:2D-Euler}. Furthermore, the integral $\int_{ \bbR_+^2 } x_{2}\omg(t,x) dx$  is also conserved by \eqref{eq:2D-Euler}, which coincides with our convention of the impulse under the sign assumption $\omg(t,\cdot) \ge 0$.

	\subsection{Lamb dipoles and the main result}\label{subsec:Lamb-Main}
	
	The collection of Lamb dipoles constitutes a two parameter family of traveling waves to \eqref{eq:2D-Euler}. 
	We set the \emph{normalized Lamb dipole} to be 
	\begin{equation}   \label{eq:lamb-norm}
		\overline{\omega}_{Lamb}(x):=  \frac{-2c_L}{    J_0(c_L)} \, J_1(c_L r) \mathbf{1}_{[0,1]}(r) \sin(\tht) 
	\end{equation}  in the polar coordinate system $(r,\tht)$. 
	Here, $J_{m}(r)$ is the $m$-th order Bessel function of the first kind and $c_L=3.83\dots$ is the first positive zero of $J_1(r)$. 
	While we defer the details to \S \ref{subsec:Lamb}, $\overline{\omega}_{Lamb}(x - t \mathbf{e}_{1})$ is a compactly supported solution to \eqref{eq:2D-Euler} whose impulse, enstrophy, and energy are given by $\pi, \pi c_{L}^{2}$, and $\pi$, respectively. See Figure~\ref{fig_lamb1} for an illustration of the vorticity and streamline of a normalized Lamb dipole in the moving frame.\footnote{The illustration is shown in the whole plane $\mathbb{R}^2$. Due to the odd-in-$x_2$ symmetry, if we only take the portion in the upper half plane, it becomes a traveling wave in $\mathbb{R}^2_+$.}

	A two-parameter family of Lamb dipoles is obtained by scaling \eqref{eq:lamb-norm}: for any $\kpp, \mu > 0$, the Lamb dipole with enstrophy $\kpp^{2}$ and impulse $\mu$ is given by 
	\begin{equation}\label{eq:lamb-rescale-def}
		\omega_{Lamb}^{\kpp,\mu}(x) :=A \,\overline{\omega}_{Lamb}\left(R^{-1}x\right), \qquad 	
		A(\kappa,\mu)= \left(\frac{\kappa^{3}}{ \sqrt{\pi} c_{L}^{3} \mu}\right)^{1/2}, \quad R(\kappa,\mu)=\left( \frac{c_{L}\mu}{\sqrt{\pi}\kpp} \right)^{1/2}, 
	\end{equation} where $A$ refers to the ``amplitude'' and $R$ ``radius'' of the rescaled Lamb dipole. With this notation, we may express the normalized Lamb dipole by $\overline{\omega}_{Lamb} = \omg_{Lamb}^{ \sqrt{\pi}c_{L} , \pi }$.

	\begin{figure}[htbp]
		\includegraphics[scale=1]{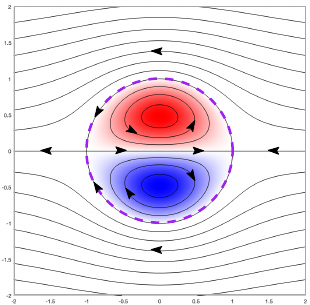} 
		\caption{Vorticity and streamlines of a normalized Lamb dipole in the moving frame} \label{fig_lamb1} 
	\end{figure}

	We are now ready to state our first main result. Given $N \ge 1$,  positive $N$-vectors $\boldsymbol{\bar{\kpp}} = \{ \bar{\kpp}_{i} \}_{i=1}^{N}$, $\boldsymbol{\bar{\mu}} = \{\bar{\mu}_{i} \}_{i=1}^{N}$ (that is, $\bar{\kpp}_{i}, \bar{\mu}_{i} > 0$ for all $1 \le i \le N$) and $\mathbf{\bar{p}} = (\bar{p}_{1},\cdots,\bar{p}_{N}) \in \bbR^{N}$, we introduce the $N$ Lamb dipole function as the linear superposition of $N$ Lamb dipoles: \begin{equation}\label{eq:N-dipole-def}
		\begin{split}
			\overline{\omg}^{\boldsymbol{\bar\kpp}, \boldsymbol{\bar\mu}}_{Lamb, \mathbf{\bar{p}}}(x) := \sum_{i=1}^{N} \omg_{Lamb}^{\bar\kpp_{i},\bar\mu_{i}}(x - \bar{p}_{i} \mathbf{e}_{1} ).
		\end{split}
	\end{equation} The velocity of the $i$-th dipole  $\omg_{Lamb}^{\bar\kpp_{i},\bar\mu_{i}}$ is given by $\bar\kpp_{i}/(\sqrt{\pi}c_{L})$ (note that $\kpp$ has the unit of velocity) and will be denoted by $\bar{V}_{i}$ for simplicity. See Figure~\ref{fig_multi_lamb} for an illustration of a superposition of three Lamb dipoles with ordered velocities $\bar V_1 > \bar V_2 > \bar V_3$.

	Given two norms $X, Y$ defined for functions on $\bbR^2_+$, we shall write $\nrm{f}_{X \cap Y} := \nrm{f}_{X} + \nrm{f}_{Y}$.

	\begin{figure}[htbp]
		\includegraphics[scale=1]{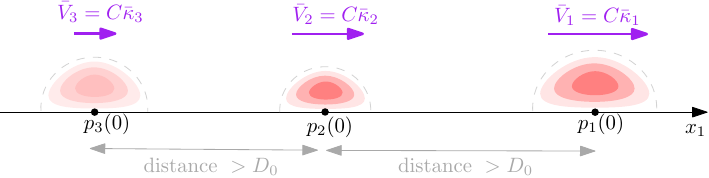} 
		\caption{\label{fig_multi_lamb} An illustration of a superposition of three Lamb dipoles with ordered velocities $\bar V_1 > \bar V_2 >\bar V_3$ that are far apart. (Note that Theorem~\ref{thm:N-Lamb-dipoles} requires the initial condition to be close to a superposition where the velocities are ordered and the centers are far apart.) } 
	\end{figure}
	
	\begin{customthm}{A}[Stability of linear superposition of $N$ Lamb dipoles]\label{thm:N-Lamb-dipoles}
		Fix $M>0$, $N\ge1$, and positive $N$-vectors $\boldsymbol{\bar{\kpp}}, \boldsymbol{\bar{\mu}}$, with $\boldsymbol{\bar{\kpp}}$ satisfying $\bar\kpp_{1} > \bar\kpp_{2} > \cdots > \bar\kpp_{N}.$  Then, there exists $\varep_{0} > 0$ depending only on $N, M, \boldsymbol{\bar{\kpp}}, \boldsymbol{\bar{\mu}}$ with the property that for any $0<\varep<\varep_{0}$, there exist $\dlt_{0}, D_{0} > 0$ depending on $\varepsilon, N, M, \boldsymbol{\bar\kpp}, \boldsymbol{\bar\mu}$ such that as long as \begin{equation*}
			\begin{split}
				\bar{p}_{i}  > \bar{p}_{i+1}  + D_{0} \qquad \mbox{for all} \qquad 1 \le i \le N - 1 , 
			\end{split}
		\end{equation*} we have the following stability result. For any initial data $\omg_{0} \ge 0$ on $\bbR^{2}_{+}$ with \begin{equation}\label{eq:ini-assumption}
			\begin{split}
				\left\Vert \omg_{0} - \overline{\omg}^{\boldsymbol{\bar\kpp}, \boldsymbol{\bar\mu}}_{Lamb, \bar{\bfp}} \right\Vert_{L^{2} \cap L^1_*} \le \dlt_{0}, \quad \nrm{\omg_{0}}_{L^{1} \cap L^{\infty}} \le M,  
				\quad\text{ and } \quad |\supp(\omega_0)|\le M, 
			\end{split}
		\end{equation} the corresponding solution $\omg(t,\cdot)$ to $\omg_{0}$ satisfies \begin{equation}\label{eq:varep-close}
			\begin{split}
				\left\Vert \omg(t,\cdot) - \overline{\omg}^{\boldsymbol{\bar\kpp}, \boldsymbol{\bar\mu}}_{Lamb, \mathbf{p}(t)} \right\Vert_{L^{2} \cap L^1_*  } \le \varepsilon \qquad\mbox{for all} \qquad t > 0 
			\end{split}
		\end{equation} for some time-dependent vector $\mathbf{p}(t) = (p_{1}(t),\cdots,p_{N}(t))$ satisfying the estimate \begin{equation}\label{eq:p-est}
			\begin{split}
				| p_{i}(t) - \bar{p}_{i} - \bar{V}_{i}\, t| \le C\varep^{1/2}(1+t), \qquad 1 \le i \le N, 
			\end{split}
		\end{equation} for some $C = C(M, N,  \boldsymbol{\bar{\kpp}}, \boldsymbol{\bar{\mu}})>0$. 
	\end{customthm}

	\begin{remark} A few remarks are in order. \begin{itemize}
			\item When $N = 1$ (single Lamb dipole case), global-in-time stability was obtained in \cite{AC2019}, and we shall discuss this case in more detail in \S \ref{subsubsec:Lamb-stability}. The ``shift estimate'' of the form \eqref{eq:p-est} was obtained later in \cite{CJ-Lamb}, with improvements in \cite{JYZ}. In the case $N = 1$, $\varepsilon_{0}$ does not depend on $\nrm{\omg_{0}}_{L^\infty}$. For $N \ge 2$, the $L^\infty$ bound of $\omg_{0}$ is necessary for estimating $\nrm{u}_{L^\infty}$ (cf. Lemma \ref{lem:vel-L-infty}), since our proof is dynamical in nature. (To this end,  a bound on $\nrm{\omg_{0}}_{L^p}$ for any $p>2$ suffices.) 
			\item When $N \ge 2$, the statement is  nontrivial even when the initial data is given \textit{exactly} as the linear superposition of Lamb dipoles. In fact, the proof in this case is not any simpler. Note that the velocity generated by each Lamb dipole is not compactly supported; it pushes down ones in the back, and lifts up ones in the front (cf. Figure \ref{fig_lamb1}). Furthermore, it is not difficult to see that the condition $\bar\kpp_{1} > \bar\kpp_{2} > \cdots > \bar\kpp_{N}$ is essential. Specifically, if a faster Lamb dipole is initially located to the left of a slower one, then by assuming \eqref{eq:varep-close} for all $t \ge 0$, it can be deduced that the faster Lamb dipole will eventually catch up to the slower one and ``collide,'' which would contradict \eqref{eq:varep-close}.\footnote{In this contradiction argument, \eqref{eq:p-est} does not need to be assumed.} Switching the initial order of the Lamb dipoles has little effect on the coercive conserved quantities of the Euler equations, clearly indicating that Theorem \ref{thm:N-Lamb-dipoles} cannot be proved solely through frozen-time variational considerations. 
			\item Note that the condition $\bar\kpp_{1} > \bar\kpp_{2} > \cdots > \bar\kpp_{N}$ only says the enstrophy (which is proportional to velocity) of the Lamb dipoles is ordered; we do not require their amplitude, impulse or radii of support to be ordered. In particular, a Lamb dipole with a larger radius is not necessarily faster than another one with a smaller radius; see the relation in \eqref{eq:lamb-rescale-def}. 
			
			\item The norm $L^{2} \cap L^1_* $ in \eqref{eq:ini-assumption}--\eqref{eq:varep-close} can be replaced by $L^{2} \cap L^1_*  \cap L^{1}$. In particular, any initial datum $\omg_{0}$ satisfying the assumptions above {gives an} example where the rescaled vorticity $t^{2} \omg(t, tx)$ {does not converge to a single Dirac delta measure} in the sense of measures as $t \to \infty$. {This is closely related to} a theory developed by Iftimie--Lopes Filho--Nussenzveig Lopes \cite{ILL2003}, see \S \ref{subsubsec:ILN} below for details.
			\item In the pioneering work of Martel--Merle--Tsai \cite{MMT02}, a similar statement which also includes the asymptotic stability was obtained for linear superpositions of the gKdV solitons. Unfortunately, the asymptotic stability for the Lamb dipoles is not known even in the case $N = 1$. One notable difference is that the family of Lamb dipoles (as well as any other traveling waves for Euler) comes in a more than two-parameter family, while gKdV solitons are just parameterized by the speed. This presents serious complications, and of course the Euler dynamics is very different from that for gKdV; see \S \ref{subsubsec:EIE} for the main challenges arising in our case.  
		\end{itemize}
	\end{remark}
	
	The strategy we develop is quite robust, and we believe that Theorem \ref{thm:N-Lamb-dipoles} can be generalized in several ways, in particular replacing each Lamb dipole with another orbitally stable one. We shall discuss potential generalizations below in \S \ref{subsec:half-plane}, after explaining the ideas and difficulties of proof. To demonstrate flexibility of the method, let us present a simple result which says that one can ``separate out'' a Lamb dipole from the solution, with the sole assumption that the other part is initially slower than the Lamb dipole under all possible rearrangements.  
	
	To state the result, given an essentially bounded $\omg \ge 0$ on $\bbR^2_+$, we define $V_{max}[\omega]$ as the maximum velocity under any rearrangement of $\omega$, namely \begin{equation*}
		\begin{split}
			V_{max}[\omg] = \sup_{g \in \calR[\omg]} \nrm{ \nb^\perp\lap^{-1}(g) }_{L^{\infty}}
		\end{split}
	\end{equation*} where $\calR[\omg]$ is the set of rearrangements of $\omg$ on $\bbR^2_+$: \begin{equation}\label{eq:rearrange}
		\begin{split}
			\calR[\omg] = \left\{ g \, : \, \left| \{ x \in \mathbb{R}^2_+ \, : \, g(x) > \lmb \} \right|  = \left| \{ x \in \mathbb{R}^2_+ \, : \, \omg(x) > \lmb \} \right| \mbox{ for all } \lmb > 0  \right\}. 
		\end{split}
	\end{equation} 
	
	\begin{customthm}{B}[Separating out a Lamb dipole from slower part] \label{thm:B}
		For any $\varep>0, \alpha \in (0,\frac{\pi}{2}]$ and $M>0$, there exist $\delta_0, D_0>0$ depending only on $\varepsilon, \alpha, M$, such that the following holds.
		
		Consider any initial data $\omega_0\geq 0$ on $\mathbb{R}_+^2$ of the form $\omega_0 = \omega_{0l} + \omega_{0r}$, where $\omega_{0r}$ satisfies the following for some $D>D_0$:
		\[
		\|\omega_{0r} - \bar\omega_{Lamb}(\cdot-D\mathbf{e}_1)\|_{L^2 \cap L^1_*  } \leq \delta_0, \quad \|\omega_{0r}\|_{L^1\cap L^\infty} \le M, 
		\quad\text{ and } \quad |\supp(\omega_{0r})|\le M, 
		\]
		and $\omega_{0l}$ satisfies $\supp\omega_{0l} \subset\{x_1 < x_2 \cot\alpha\}$, and $V_{max}[\omega_{0l}] < \sin\alpha$.  
		Then, with $V_{avr}$ and $\omega_r(t,\cdot)$ defined by \begin{equation*}
			\begin{split}
				V_{avr} := \frac{1}{2} \left(1+\frac{ V_{max}[\omega_{0l}]  }{\sin\alpha} \right)<1 \quad\text{ and }\quad \omega_r(t,\cdot):= \omega(t,\cdot) \mathbf{1}_{\{x_{1} > x_{2}\cot\alpha + V_{avr} t\}},
			\end{split}
		\end{equation*} we have the following stability: for some $p(t)$ satisfying $|p(t)-t-D|\le C\varepsilon^{1/2}(1+t)$,
		\[ \| \omega_r(t,\cdot) - \bar\omega_{Lamb}(\cdot-p(t) \mathbf{e}_1)\|_{L^2\cap L^1_* } \le \varep \quad \text{ for all }  \quad t \ge 0. 
		\]  
	\end{customthm}

	\begin{center}
		\begin{figure}[htbp]
			\includegraphics[scale=1]{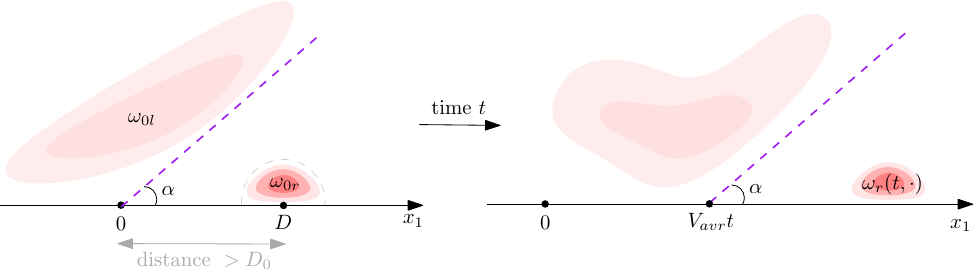} 
			\caption{\label{fig_thmB}An illustration of the initial data of Theorem~\ref{thm:B} (left figure), and the solution at time $t$ (right figure). Here $\omega_{0r}$ and $\omega_{r}(t,\cdot)$ are the vorticities on the yellow background at time $0$ and $t$ respectively.} 
		\end{figure}
	\end{center}
	See Figure~\ref{fig_thmB} for an illustration of the assumptions and conclusions of Theorem~\ref{thm:B}.
	\begin{remark}\label{rem:B}
		We give a few remarks regarding Theorem \ref{thm:B}. \begin{itemize}
			\item The point of this result is that we do not require any ``detailed'' information on the part $\omg_{0 l}$ for separating out a Lamb dipole to the right. It is expected that if furthermore $\omg_{0 l}$ is close to any of the ``locally energy maximizing'' dipole, then this information also propagates for all $t \ge 0$; we leave this extension to a future work.
			\item This result is stated for simplicity with the normalized Lamb dipole $\bar\omg_{Lamb}$ defined in \eqref{eq:lamb-norm}, and a similar result holds for all the other Lamb dipoles by appropriately rescaling the assumptions. 
			\item The condition $V_{max}[\omega_{0l}] < \sin\alpha$ is sharp in view of particle trajectories: any fluid particle starting from the support of $\omg_{0 l}$ is not able to reach the line $\{x_1 = x_2 \cot\alpha + V_{avr}t \}$ for $t \ge 0$ by the velocity that $\omg_{0 l}$ and its all possible rearrangements generate.  
			\item It is well-known that 
			$V_{max}[\omega] \le C_{0}\|\omega\|_{L^\infty}^{1/2}\|\omega\|_{L^1}^{1/2}$ holds 
			for some universal constant $C_{0}>0$ \cite{ILL2003,Iftimie1}. Therefore, Theorem \ref{thm:B} is applicable if  $\omg_{0 l} \in L^{\infty}$ is close to a vortex patch $\mathbf{1}_{ \Omg_{0} }$ with $\Omg_{0} \subset\{x_1 < x_2 \cot\alpha\}$  and $|\Omg_{0}|$ sufficiently small. This also shows that for Theorem \ref{thm:B} to hold, it is not necessary for the slower part $\omg_{0 l}$ to be located to the ``left'' of the Lamb dipole. 
		\end{itemize}
	\end{remark}

	\subsection{Overall strategy and difficulties}\label{subsec:difficulty}
	
	\subsubsection{Stability of a single Lamb dipole}\label{subsubsec:Lamb-stability}
	
	The first ingredient in our proof is the stability for a single Lamb dipole from \cite{AC2019,ACJ}. 
	To describe the setup, we introduce the admissible class 
	\begin{equation}\label{eq:admissible}
		\widetilde{\mathcal{A}}_{\kappa,\mu}:=\left\{
		\omega \ge 0 \mbox{ on } \bbR^2_+ \,:\, 
		\, \nrm{\omg}_{L^{2}(\bbR^2_+)} \le \kpp,\,\, \nrm{\omg}_{L^1_* (\bbR^2_+)} \le \mu  
		\right\}.
	\end{equation} Note that  $\omg_{Lamb}^{\kpp,\mu}\in \widetilde{\mathcal{A}}_{\kappa,\mu}$ by definition, and the main result of \cite{AC2019} (see Theorem \ref{thm:AC2019} below) says that $\omg_{Lamb}^{\kpp,\mu}$ is the energy maximizer in the class \eqref{eq:admissible}, unique up to shifts along the $x_{1}$-axis. This variational characterization of the Lamb dipoles implies Lyapunov stability for each of them. In this work, we are going to use the following version, which says that under tight $L^{2}, L^1_* $ bounds, any vorticity having similar energy with the corresponding Lamb dipole must be close to it. 
	\begin{proposition}\label{prop:Lamb-energy-estimate}
		Fix some $\kpp,\mu>0$ and $M > 0$. Then, for any $\varep_{1}>0$, there exists $\dlt_{1} = \dlt_{1}(\kpp,\mu,M,\varep_{1})>0$ such that the following holds: if $\omg \ge 0$, $\nrm{\omg}_{L^{1}(\bbR^2_+)} \le M$ satisfies  \begin{equation*}
			\begin{split}
				\omg \in \widetilde{\mathcal{A}}_{ \kpp +\dlt_{1}, \mu +\dlt_{1} } \qquad\mbox{and}\qquad E[\omg] \ge E[\omg_{Lamb}^{\kpp,\mu}] - \dlt_{1}, 
			\end{split}
		\end{equation*} then we have \begin{equation*}
			\begin{split}
				\inf_{\tau\in\bbR} \nrm{ \omg - \omg_{Lamb}^{\kpp,\mu}(\cdot - \tau \mathbf{e}_{1}) }_{L^{2} \cap L^1_*  (\bbR^2_+)} \le \varep_{1}.  
			\end{split}
		\end{equation*} 
	\end{proposition}
	The proof of this proposition is given in \cite[Proposition 2.6]{CJY}, but again, this is just a reformulation of the compactness statement from \cite[Theorems 1.3 and 1.5]{AC2019}.
	
	\subsubsection{Decomposing the solution into $N$ pieces} To prove Theorem \ref{thm:N-Lamb-dipoles}, we shall decompose the solution $\omg(t,\cdot)$ appropriately into $N$ disjoint pieces, and apply Proposition \ref{prop:Lamb-energy-estimate} to each of them. There does not seem to be a ``canonical'' way to decompose the solution, and we simply used hard cutoffs: we take (for the precise definition, see \S \ref{subsec:decomp-N-vortices}) \begin{equation}\label{eq:decomp-P}
		\begin{split}
			\omg(t,\cdot) = \sum_{i=1}^{N} \omg_{i}(t,\cdot), \qquad \omg_{i}(t,\cdot) := \omg(t,\cdot) \mathbf{1}_{ \left\{  \frac{P_{i}(t) + P_{i+1}(t)}{2} \le x_{1} < \frac{P_{i-1}(t) + P_{i}(t)}{2} \right\} }(\cdot)
		\end{split}
	\end{equation} where $P_{i}(t) = \bar{p}_{i} + C_{L} \bar\kpp_{i}t = \bar{p}_{i} + \bar{V}_{i}t$ is the ``expected location'' of the $i$-th Lamb dipole for $1 \le i \le N$ when interactions are absent, and $P_{0}:=+\infty$ and $P_{N+1}:=-\infty$. We shall set $K_{i}(t) := \nrm{\omg_{i}(t)}_{L^{2}}^{2}$, $\mu_{i}(t) := \nrm{\omg_{i}(t)}_{L^1_* }$, and $E_{i}(t) := E[\omg_{i}(t)]$ for simplicity.  Note that enstrophy and impulse are both \textit{additive}: we have \begin{equation}\label{eq:conserved-sum}
		\begin{split}
			K_{1}(t) + \cdots + K_{N}(t) = \nrm{\omg_{0}}_{L^{2}}^{2} \quad\mbox{ and }\quad \mu_{1}(t) + \cdots + \mu_{N}(t)  = \nrm{\omg_{0}}_{L^1_* }
		\end{split}
	\end{equation} for all $t$. 
	On the other hand, the energy has a positive ``interaction'' term $E_{inter}(t)$ (see \S\ref{subsubsec:energy} for the precise definition), so that the total energy satisfies $$E_{0} = E(t) = E_{1}(t) + \cdots + E_{N}(t) + E_{inter}(t) > E_{1}(t) + \cdots + E_{N}(t).$$


For Proposition \ref{prop:Lamb-energy-estimate} to be applicable to each of  $\omg_{i}(t,\cdot)$, we need \textit{global-in-time} bounds \begin{equation}\label{eq:goal}
	\begin{split}
		K_{i}(t)   \le K_{i}(0) + C\dlt_{0}, \quad \mu_{i}(t)   \le \mu_{i}(0) + C\dlt_{0}, \quad  E_{i}(t) > E_{i}(0) - C\dlt_{0}
	\end{split}
\end{equation} for all $1 \le i \le N$, with some time-independent $C>0$. We emphasize that $\dlt_{0}$-smallness is needed (which is from initial data) and not just $\varep$-smallness, which comes from the conclusion \eqref{eq:varep-close}.\footnote{Even in the single Lamb dipole case $N = 1$, the quantitative relation between $\dlt_{0}$ and $\varepsilon$ is unknown, and in any case we expect $\dlt_{0} \ll \varep$, so we cannot rely on this $\varepsilon$-smallness to close estimates of the form \eqref{eq:goal}.} 	The immediate difficulty we face in applying Proposition \ref{prop:Lamb-energy-estimate} with each $\omg_{i}(t,\cdot)$ is that all of $K_{i}, \mu_{i}, E_{i}$ are time-dependent due to interactions. In fact, as we shall explain below, there are competing effects which could make each of $K_i(t)$ and $\mu_i(t)$ grow or decay.

\subsubsection{Point vortex case} To see whether \eqref{eq:goal} has any chance to work, it is natural to first consider the case when each Lamb dipole is replaced by a point vortex. This was indeed considered in the pioneering work \cite{ILL2003}. To be concrete, take an initial data of the form \begin{equation}\label{eq:ini-PV}
	\begin{split}
		\omg_{0} = \sum_{i = 1}^{N} \Gmm_{i} \dlt_{ X_{i}(0) } \mbox{ on } \bbR^2_+, \quad X_{i}(0) = ( p_{i}(0), h_{i}(0) ), 
	\end{split}
\end{equation} with $\Gmm_{i}, h_{i}(0) > 0$ and $p_{1}(0) > \cdots > p_{N}(0)$. In the absence of interactions, the $i$-th vortex translates horizontally to the right with uniform speed $V_{i}  := \Gmm_{i}/(4\pi h_{i}(0))$. Similarly with the assumptions of Theorem \ref{thm:N-Lamb-dipoles}, we take \begin{equation}\label{eq:ini-PV-assume}
	\begin{split}
		V_{1} > V_{2} > \cdots > V_{N}, \qquad p_{i}(0) - p_{i+1}(0) > D_{0} \gg 1, \quad 1 \le i \le N - 1. 
	\end{split}
\end{equation} If we now let \eqref{eq:ini-PV} evolve according to the $N$ vortex system on $\bbR^2_+$ (Kirchhoff--Helmholtz equation) under the assumptions \eqref{eq:ini-PV-assume}, then it can be proved that the corresponding solution \begin{equation*}
	\begin{split}
		\omg(t,\cdot) = \sum_{i = 1}^{N} \Gmm_{i} \dlt_{ X_{i}(t) } \qquad X_{i}(t) = ( p_{i}(t), h_{i}(t) ), 
	\end{split}
\end{equation*} is global in time and satisfies  \begin{equation}\label{eq:stability-PV}
	\begin{split}
		|h_{i}(t) - h_{i}(0)| \ll 1, \qquad \frac1{t}|p_{i}(t) - p_{i}(0) - V_{i}t| \ll 1
	\end{split}
\end{equation} for all $1 \le i \le N$ and $t \ge 0$.\footnote{{In \cite{ILL2003}, the proof was given only in the case $N = 2$. In the case of general $N$, probably the simplest way to prove it is to repeat the bootstrapping scheme used in the current work. Filamentation, which is the main enemy in the case of Lamb dipoles, is completely absent for point vortices and the proof is significantly simpler.}} Therefore, an analogue of \eqref{eq:goal} holds for sufficiently separated point vortices, although we do not really have a concept corresponding to the enstrophy in this case. The key observations which give \eqref{eq:stability-PV} are: if we take $D_{i}(t)$ to be the smallest distance from $X_{i}(t)$ to the other point vortices in $\bbR^2_+$, we have the interaction velocity decay $\lesssim (D_{i}(t))^{-2}$, as well as linear in time separation $D_{i}(t)- D_{i}(0) \gtrsim t$, which makes the integrated (from $t = 0$ to $+\infty$) effect of mutual interactions small, as long as $D_{i}(0)$ is large. This type of velocity decay and linear separation, whose precise version is given in Lemma \ref{lem:vel-decay} below, is indeed the key mechanism behind our main results as well. {However, despite the above analogue, the setting of our main results is different from the desingularization regime: the Lamb dipoles have $O(1)$ energy and enstrophy as long as they have bounded velocity and circulation (integral of the vorticity), while energy and enstrophy blow up in the desingularization regime.
}

\subsubsection{Enstrophy, Impulse, and Energy} \label{subsubsec:EIE}

We are in a position to discuss the key technical issues in establishing \eqref{eq:goal}, most of which are present already in the case $N = 2$. In this simplest case, we use the borderline $L(t) := \frac12( \bar{p}_1+\bar{p}_2+(\bar{V}_1+\bar{V}_2)t )$ to decompose $\omega(t,\cdot)$ as \begin{equation*}
	\begin{split}
		\omg(t,x) = \omg_{1}(t,x) + \omg_{2}(t,x) := \omg(t,x) \mathbf{1}_{ \{ x_{1} \ge L(t) \} } + \omg(t,x) \mathbf{1}_{ \{ x_{1} < L(t) \} }, 
	\end{split}
\end{equation*} and the goal is to obtain \eqref{eq:goal} for $i = 1, 2$.

\medskip

\noindent \textbf{Evolution of enstrophy and impulse}. In the simplest case of $N = 2$, thanks to the conservation law \eqref{eq:conserved-sum}, it suffices to prove $\int_0^\infty|\dot{K}_{1}(t)|dt, \int_0^\infty|\dot{\mu}_{1}(t)|dt \lesssim  \dlt_{0}$, where $\dot{}$ denotes the time derivative. A direct computation yields\begin{equation*}
	\begin{split}
		\dot{K}_{1}(t) = - \int_{ \{ x_{1} = L(t) \} } (\dot{L} - u_{1}(t,x) ) (\omg(t,x))^2 dx_{2}
	\end{split}
\end{equation*} 
and
\begin{equation*}
	\begin{split}
		\dot{\mu}_{1}(t) = - \int_{ \{ x_{1} = L(t) \} } (\dot{L} - u_{1}(t,x) ) x_{2} \omg(t,x) dx_{2} + \int_{ \{ x_{1} > L(t)  \} } (u_{2})_{2}(t,x)\omg_{1}(t,x) dx , 
	\end{split}
\end{equation*}  and the bootstrap assumptions give that $\dot{L} \gg |u_{1}|$ on the line $\{ x_{1} = L(t) \}$ for all $t \ge 0$. (Here, $(u_{2})_{2}$ is the second component of the velocity generated by $\omg_{2}$.) Therefore, the first terms in the right-hand sides are always negative, which represents flux of enstrophy and impulse from right to left through the border $\{ x_{1} = L(t) \}$. 

\medskip

\noindent \textbf{Filamentation and variational principle for two Lamb dipoles}. These flux terms capture the effect of \textit{filamentation}, or creation of long tails, behind a Lamb dipole.\footnote{This behavior is not specific to Lamb dipoles and widely observed in experiments and simulations \cite{KrXu,MoMo}.  Some rigorous results in formation of filaments are given in \cite{CJ-Lamb,CJ_Hill,DEJ,JYZ}.} Filamentation occurs \textit{generically} because, in the reference frame moving with a Lamb dipole, its velocity field advects vorticity perturbations from the front to the rear outside the dipole, see Figure \ref{fig_lamb1}. This is the immediate difficulty we face, compared to the point vortex case. 

Even in the single Lamb dipole case ($N=1$), the only known way to rule out large filamentation is through the energy conservation and appealing to the energy maximizing character of a single Lamb dipole. Assuming for a moment that the last term in the equation for $\dot{\mu}_1$ is negligible, for $N = 2$, we need to have a statement that large filamentation from $\omg_1$ to $\omg_2$ is \textit{energy unfavorable}, namely it decreases the total energy of the solution. To be more concrete, one may introduce the admissible class involving two ``infinitely separated'' vortices, for instance\footnote{This type of admissible classes were used for both construction and stability of ``multi-soliton'' solutions in certain situations in incompressible and compressible flows; see for instance \cite{JaSe,CQZZ.2023}.} \begin{equation*} 
	\begin{split}
		\qquad  \widetilde{\mathcal{A}}^{(2)} := \left\{ \omg = \omg_{1} + \omg_{2} \, : \, \dist(\supp(\omg_{1}), \supp(\omg_{2})) \gg 1,  \, \omg_{1} \in \widetilde{\mathcal{A}}_{\bar\kpp_1,\bar\mu_1},  \, \omg \in  \widetilde{\mathcal{A}}_{(\bar\kpp_1^2 + \bar\kpp_2^2)^{1/2},\bar\mu_1+\bar\mu_2} \right\}. 
	\end{split}
\end{equation*} Somewhat surprisingly, this approach simply cannot work,  because $\bar{\omg}_1 + \bar{\omg}_2$ with $\bar{\omg}_{i} := \omg_{Lamb}^{ \bar\kpp_i, \bar\mu_i }$ is not necessarily a local energy maximizer in  $\widetilde{\mathcal{A}}^{(2)}$, even when $\bar\kpp_1 > \bar\kpp_2$. Local energy maximizing property holds only if we assume additionally  ${\bar\kpp_{2}}/{\bar\mu_{2}} > {\bar\kpp_{1}}/{\bar\mu_{1}}$. In other words, when $\bar\kpp_1 > \bar\kpp_2$ and ${\bar\kpp_{2}}/{\bar\mu_{2}} \le  {\bar\kpp_{1}}/{\bar\mu_{1}}$, we are in the ``energy favorable'' regime, in the sense that pure backward transfer of enstrophy and impulse either preserves or even increases the total energy! This suggests the possibility that over time, $\omega_{2}$ may accelerate and eventually surpass $\omega_{1}$, breaking the stability without violating any of the conserved quantities. 

\medskip

\noindent \textbf{Relating energy variation with impulse variation}. It turns out that multiple Lamb dipole stability holds even in the above ``energy favorable'' regime: our key observation is that the individual energy tends to vary much slowly relative to the corresponding impulse: The time derivative of $E_{1}(t) := E[\omg_{1}(t,\cdot)]$ satisfies a similar identity with that of $\dot{\mu}_{1}$, and its negative term (due to filamentation) turns out to be much smaller than the corresponding negative term in $\dot{\mu}_{1}$.\footnote{This estimate quantifies the well-known tendency of incompressible fluids, at least within this specific setting: once energy becomes concentrated in a region, it tends to remain localized.} With some algebraic computations, this observation rules out stability breaking by large filamentation.

\medskip 

\noindent \textbf{The ``lift-up'' effect and  Lagrangian analysis of Gain--Error interaction}. We still need to show that the last term in the equation for $\dot{\mu}_1$ is small, after integrated in time from $0$ to $\infty$. This last term is actually nonnegative: $(u_2)_2(t,x) > 0$ pointwise on $\{ x_{1} > L(t) \}$, namely all the fluid particles from $\omg_{2}$ are ``lifting up'' those in $\omg_{1}$, increasing the impulse of $\omg_{1}$.\footnote{This can be also understood from the streamlines depicted in Figure \ref{fig_lamb1}.}  Strictly away from the borderline $L(t)$, this lift-up effect decays with rate $t^{-2}$ in time, which makes the time-integrated effect small for initially far away Lamb dipoles. This is indeed the main lesson from the case of point vortices. 

However, what may happen is that, close to the borderline $L(t)$, the part of the vorticity that $\omg_{2}$ gains from $\omg_{1}$ via filamentation may continuously transfer its impulse back to the ``adjacent'' part which is still contained in $\omg_{1}$.  We shall refer to this scenario as \textbf{Gain--Error interaction}: controlling this is the most technically involved part of the current paper. Let us summarize the key issues here: \begin{itemize}
	\item \textbf{Error} is defined by the part of $\omg_{1}$ which is relatively close to the borderline with $\omg_{2}$; its smallness is given in Lemma \ref{lem:omg-err-small}.  
	\item \textbf{Gain} is defined by the part of $\omg_{2}$ which came originally from $\omg_{1}$; its smallness is deduced in Lemma \ref{lem:gain-bounds}. To control the Gain--Error interaction, we need to have the $L^{1}$ control over the gain part--when two fluid particles are very close, the kernel defining the impulse exchange is singular, and cannot be directly controlled either by the enstrophy nor impulse.  
\end{itemize}
To deal with this issue of having singular kernel, we switch to Lagrangian coordinates: to get uniform smallness integrated over time, we track the whole trajectory of each ``gain particle'' and explicitly show that it can make a nontrivial contribution only for a relatively short period of time during its ``lifetime,'' which is sufficient since the total circulation ($L^{1}$ norm) of gain particles is small (Lemma \ref{lem:omg-err-small}). This key Lagrangian computation is presented in Lemma \ref{lem:gain-error}. In the end, what we manage to show is that the total impulse gain (integrated from $t = 0$ to $\infty$) of $\omg_{1}$ from $\omg_{2}$ can be bounded in terms of the total impulse loss of $\omg_{1}$ to $\omg_{2}$ via filamentation.

\medskip 

\noindent \textbf{Lagrangian bootstrapping scheme and the shift estimate}. The previous strategy needs to be carried out carefully through a set of bootstrap assumptions: we assume \eqref{eq:goal} for some $C\dlt_{0}$ over an interval of time $[0,T]$ and  proceed to show, using these bootstrap assumptions, that \eqref{eq:goal} actually holds for $C\dlt_{0}/100$, say.\footnote{We remark again that we do not have a quantitative relation between $\varep$ and $\dlt_{0}$, and in particular $\varep$-closeness of the solution to a sum of Lamb dipoles (which is the conclusion of Theorem \ref{thm:N-Lamb-dipoles}) can be used only in a very limited way.} In this process, we also need to ensure that each piece $\omg_{i}$ defined by \eqref{eq:decomp-P} ``contains'' the corresponding Lamb dipole near the expected location. This does not follow from energy considerations, and the estimate which tracks the approximate location of the Lamb dipole is based on another Lagrangian bootstrapping scheme itself. In this work, we follow the strategy from \cite{JYZ} for this part of the argument.

\medskip 

\noindent \textbf{Functional inequalities involving enstrophy and impulse}. Since our bootstrap assumptions provide sharp bounds on the enstrophy and impulse, it is very important to control the solution based on these two quantities as much as possible. For instance, Proposition \ref{prop:Energy-X} shows that the interaction energy can be bounded purely in terms of the enstrophy and impulse. While the proof is elementary, such a bound seems to be new.

\medskip

In the very end, the reason why Theorem \ref{thm:N-Lamb-dipoles} works is because the linear separation between the Lamb dipoles is sufficiently fast to prevent filaments from gaining and exchanging enough circulation, enstrophy, impulse, and energy with each other. The proof of Theorem \ref{thm:B} is more or less similar, with a slight difference in the choice of the bootstrap assumptions.

\subsection{Dynamics in the upper half-plane}\label{subsec:half-plane}

Let us review relevant works regarding dynamics of the Euler equations in the upper half-plane, to put our main results in context.

\subsubsection{General results on the upper half-plane} \label{subsubsec:ILN}

There are a few results concerning the general long time dynamics of nonnegative vortices in $\bbR^2_+$ \cite{I99,ISG99,ILL2003,Iftimie1}. In \cite{I99}, Iftimie proved that for any nontrivial nonnegative $\omg_{0} \in L^{1} \cap L^{\infty}(\bbR^2_+)$, there exists $c_{0} > 0$ such that for all time, \begin{equation*}
	\begin{split}
		\int_{\bbR^2_+} x_{1}\omg(t,x) dx \ge 	\int_{\bbR^2_+} x_{1}\omg_{0}(x) dx + c_{0}t 
	\end{split}
\end{equation*} where $\omg(t,\cdot)$ is the corresponding Yudovich solution.\footnote{The linear rate of growth is sharp since there is a uniform in time bound for $\nrm{u(t,\cdot)}_{L^\infty}$. While this result does not imply that the support of the vorticity in the $x_{1}$-coordinate can also grow linearly in time, this was proved later in \cite{CJ-Lamb} by filamentation from a single Lamb dipole.} In the same paper and in \cite{ILL2003}, it was further proved that for any $t \ge 0$, 
\begin{equation*}
	\begin{split}
		x_{2} \le C(t\log (2+t))^{1/3}, \qquad x_{1} \ge -D(t \log(2+t))^{1/2} \qquad \mbox{for all} \quad x = (x_1,x_2) \in \supp(\omg(t,\cdot)), 
	\end{split}
\end{equation*}
with the additional assumption that $\omg_{0}$ is compactly supported, where $C, D>0$ depend only on $\omg_{0}$. To the best of our knowledge, it is still not known whether these inequalities are sharp. Furthermore, \cite{ILL2003} obtained a \textit{scattering result} upon rescaling: defining $\tld{\omg}(t,x) := t^{2} \omg(t, tx)$ (this preserves the $L^{1}$ norm for each $t>0$), it was proved that if $\tld{\omg}(t,\cdot)$ (viewed as a measure on $\bbR^2_+$) weakly converges to a measure $\mu \otimes \dlt_{0}$ as $t \to \infty$, then $\mu$ must be the sum of an at most countable set of Dirac deltas, which can possibly accumulate only at 0. That is (we refer to \cite{ILL2003} for the precise statement), \begin{equation*}
	\begin{split}
		t^{2} \omg(t, tx) \rightharpoonup  \sum_{i = 1}^{N} \Gmm_{i} \dlt(x_{1} - a_{i}) \otimes \dlt(x_{2})
	\end{split}
\end{equation*} where $1 \le N \le \infty$, $\Gmm_{i} > 0$ and $a_{i} \in [0,C]$ where $C = \sup_{t\ge0}\nrm{ u_{1}(t,\cdot)}_{L^{\infty}}<\infty$.

{It is likely that} the solutions of Theorem \ref{thm:N-Lamb-dipoles} give a nontrivial limiting measure, where for each $1 \le i \le N$, $\Gmm_{i}$ and $a_{i}$ are given as the $L^{1}$ norm and the traveling speed of the $i$-th Lamb dipole, up to small errors uniform in $t>0$. {To prove this based on Theorem \ref{thm:N-Lamb-dipoles}, one still needs to remove the $O(t)$ term in the right-hand side of \eqref{eq:p-est}, after potentially re-defining $p_{i}(t)$.} Given that the result of \cite{ILL2003} does not exclude the case of countably many Dirac deltas accumulating at the origin, it is tempting to speculate stability for a countable sequence of rescaled Lamb dipoles, summable in the $L^{1}$ norm. However, this seems to be a challenging problem for now. 

\subsubsection{Traveling waves on the upper half-plane} There is a vast literature on the study of traveling wave solutions of \eqref{eq:2D-Euler} on the upper half-plane, many of which concern \textit{existence}, the Lamb dipole being a rare example with an explicit formula. There are relatively few results proving the nonlinear stability of such traveling waves defined on $\bbR^2_+$ \cite{Wan86,Burton05b,AC2019,Wang.2024,CQZZ25,Tu83,Tur87,BNL13,Burton21,CSW,Burton88,Burton89,Burton96}. These stability results are largely based on the expectations of Kelvin \cite{Kelvin1880}, Arnol'd \cite{Arnold66}, and Benjamin \cite{Ben84} that stable traveling waves are those who maximize the kinetic energy under the impulse constraint. This variational principle can be supplemented with additional constraints which are conserved by the Euler flow, e.g. the $L^p$ norm of the vorticity and more generally the distribution function. The stability result for such traveling waves with general distribution function was obtained by Burton, Nussenzveig Lopes, Lopes Filho in \cite{BNL13}, with relatively recent improvements in \cite{Burton21,Wang.2024}. These works concern the case when the vortex is separated from the boundary of the half-plane. More recently, stability for a larger class of traveling waves including the Lamb dipole was obtained in \cite{AC2019}. Furthermore, the \textit{uniqueness} of the maximizer under appropriate constraints is open, except for the case of Lamb dipoles \cite{Burton05b,AC2019} and the ``large impulse regime'' where the vortex is asymptotically circular \cite{CQZZ25}. Without uniqueness, the stability needs to be stated for the whole set of maximizers. 

Numerical and experimental works \cite{OvZa82,Pul92,VV98,KrXu,GF89,FV94} demonstrate the existence and stability of even a larger class of traveling waves than what is known. 

We expect that the methods used here could be generalized to obtain stability of linear superpositions of other traveling waves, under the conditions that (i) each traveling wave is orbital stable in an appropriate sense, (ii) faster traveling wave is always located to the right relative to slower ones, and (iii) they are initially sufficiently far apart. In particular, a natural generalization is to consider the Sadovskii vortex \cite{HT,CJS,CSW}, which is a traveling wave patch that admits a similar variational characterization to the Lamb dipole.  Still, many technical challenges will be present, especially when each of them is not related to the other by rescaling of some common profile. 

Although we only consider 2D Euler on $\bbR^2_+$ in this work, an analogous problem is to consider single-signed vortex rings in $\bbR^3$, which are axisymmetric solutions to the 3D Euler equations \cite{Wan88,Burton-ring,Choi-Hill,CQYZZ}. This is indeed the original setup for Kelvin and Benjamin, which motivated the study of traveling waves on $\bbR^2_+$ as a model problem. 

Lastly, we note that there are several recent works which concern long time dynamics of vortices in the strongly interacting regime, most notably the leapfrogging motion, both in two and three dimensions \cite{HHM,DdPMW,DdPMW22,BCM,DHLM}. 

\subsubsection{Case of the quadrant: additional odd symmetry} In the half-plane, one may consider vorticities that are odd in $x_{1}$, which reduces the motion to the first quadrant $(\bbR_+)^2$. If one further assumes that the vorticity is nonnegative on $(\bbR_+)^2$, then one obtains monotonicity (\cite{ISG99})\begin{equation*}
	\begin{split}
		\frac{d}{dt} \int_{ (\bbR_+)^2 } x_{1} \omg(t,x) dx > 0 > 	\frac{d}{dt} \int_{ (\bbR_+)^2 } x_{2} \omg(t,x) dx . 
	\end{split}
\end{equation*} Under these symmetry and sign assumptions, the stability ($t\ge0$) for two odd-symmetric Lamb dipoles moving away from each other was obtained in \cite{CJY}. As this formula shows, the vorticity impulse restricted to $(\bbR_+)^{2}$ decreases with time. In particular, one does not need to worry about impulse gain, which is the most challenging issue in the current work. 

Based on a detailed asymptotic information for sharply concentrated traveling waves in the half-plane (\cite{DdPMP}), the work \cite{DdPMP2} obtained the existence of a concentrated vortex in $(\bbR_+)^{2}$, which is asymptotic to a traveling wave in $\bbR^2_+$ as $t\to+\infty$. It would be an interesting problem to obtain the Lamb dipole in this asymptotic limit as well.

\subsubsection{Viscous traveling waves} 

Another very important issue is to understand the dynamics of slightly viscous traveling waves, see numerical computations in \cite{KrXu}. Very recently, \cite{DoGa} obtained a long time approximation of a viscous traveling wave with a pair of Dirac deltas on $\bbR^2$ as the initial condition; see \cite{GaSv} for the case of a single vortex ring.

\subsection{Organization of the paper}

The remainder of the paper is organized as follows. In \S \ref{sec:prelim}, we collect several inequalities involving the stream function, velocity, and energy. We also present some computations for the Lamb dipole and  discuss its stability. The proof of Theorem \ref{thm:N-Lamb-dipoles} is carried out through \S \ref{sec:overview}, \S \ref{sec:boot-con}, and \S \ref{sec:boot}. In \S \ref{sec:overview}, we introduce various types of decompositions of the solution, and fix the  bootstrap hypotheses. Various consequences of the bootstrap hypotheses are established then in \S \ref{sec:boot-con}. Then, we close the bootstrap assumptions in \S \ref{sec:boot}, completing the proof of Theorem \ref{thm:N-Lamb-dipoles}. Lastly, we prove Theorem \ref{thm:B} in \S \ref{sec:B}. 

\bigskip

\noindent \textbf{Acknowledgments}. KA is supported by the JSPS through the Grant in Aid for Scientific Research (C) 24K06800 and MEXT Promotion of Distinctive Joint Research Center Program JPMXP0619217849. IJ has been supported by the NRF grant from the Korea government (MSIT), No. 2022R1C1C1011051, RS-2024-00406821, the Asian Young Scientist Fellowship, and the KIAS Individual Grant at Korea Institute
for Advanced Study.  YY has been supported by the NUS startup grant, MOE Tier 1 grant, and the Asian Young Scientist Fellowship. 

{We are very grateful for the anonymous referees for their careful reading of the manuscript and providing several helpful remarks, which have been incorporated into this current version. We also thank Kyudong Choi, Kwan Woo, Guolin Qin, and Chanhong Min for their comments.}

\section{Preliminaries}\label{sec:prelim}

In this section, we collect several useful inequalities which will be used later in the proofs. Most of the inequalities are well-known (cf. \cite{Iftimie1,ILL2003,Burton88}), but the ones in Proposition \ref{prop:Energy-X} are new, to the best of our knowledge.

\subsection{Notations and conventions} 

For a point $x = (x_1,x_2) \in \bbR^2$, we write $\bar{x} = (x_1,-x_2)$. With this notation, for two points $x, y \in \bbR^2$, we have $|x-\bar{y}| = |\bar{x}-y| \ge |x_2+y_2|$. The integrals and norms will be taken in the upper half-plane $\bbR^2_+$, unless otherwise specified. The general bounds for the stream function, velocity, and energy do not require that $\omg$ is nonnegative, and we shall mention it explicitly when the assumption $\omg\ge0$ is necessary.

\subsection{Bounds for stream function, velocity, and energy}

\subsubsection{Stream function} Given a vorticity $\omg$, the corresponding stream function $\psi = \lap^{-1}\omg$ is given explicitly by \begin{equation}\label{eq:G-def}
	\begin{split}
		\psi(x) 
		& = -\frac{1}{4\pi} \int_{\mathbb{R}^2_{+}} \log \left( 1 + \frac{4x_{2}y_{2}}{|x-{y}|^{2}} \right)  \, \omega(y)\,dy =: -\int_{\mathbb{R}^2_{+}} \bfG(x,y) \, \omega(y)\,dy .
	\end{split}
\end{equation} The kernel $\bfG$ is nonnegative for all $x, y \in \bbR^2_{+}$ and therefore if $\omg\ge0$ on $\bbR^2_+$, we have that $\psi\le 0$. Furthermore, for any $0<\alp\le1$, there is a constant $C_\alp>0$ depending only on $\alp$ such that 
\begin{equation}\label{eq_alpha}
	\begin{split}
		0\leq \bfG(x,y) \le C_{\alp}  \frac{(x_2y_2)^\alp}{|x-{y}|^{2\alp}} \quad\text{ for all }x,y\in\bbR^2_+, x\neq y.
	\end{split}
\end{equation} {This follows from the general inequality $\log(1+z) \le C_{\alp}z^{\alp}$, which works for any $z \ge 0$ and $0<\alp \le 1$.} In particular, if $x\in\mathbb{R}^2_+$ satisfies $x_2>0$, and $x$ is separated from the support of $\omega$ by a distance $D(x) = \inf_{y \in \supp(\omg)} |x-{y}| > 0$,  we can apply the above inequality for $\alpha=1$ and obtain
\begin{equation*}
	\begin{split}
		\frac{|\psi(x)|}{x_{2}} \le \frac{C}{D^2(x)}\nrm{ \omg}_{L^1_* }. 
	\end{split}
\end{equation*}

\subsubsection{Velocity } The velocity $u = (u_1,u_2)$ is defined by $u(x) = \nb^\perp\psi =   -\nb^\perp(-\lap)^{-1}\omg$. For convenience, we introduce the kernel 	$\bfK(x,y) = (\bfK_1(x,y),\bfK_2(x,y)) = -\nb^\perp_{x}\bfG(x,y)$ so that \begin{equation}\label{eq:K-def}
	\begin{split}
		u(x) = \int_{\bbR^2_+} \bfK(x,y) \omg(y) dy . 
	\end{split}
\end{equation} 

We present a simple lemma for the velocity.  
\begin{lemma}\label{lem:vel-decay}
	The velocity kernel $\bfK$ satisfies \begin{equation}\label{eq:K-decay}
		\begin{split}
			|\bfK(x,y)| \le \frac{Cy_{2}}{|x-y||x-\bar{y}|}.
		\end{split}
	\end{equation} If $x\in\mathbb{R}_2^+$ is separated from $\supp(\omega)$ by a distance $D(x) :=  \inf_{y \in \supp(\omg)} |x-y| > 0$, then $u(x)$ satisfies \begin{equation}\label{eq:u-decay}
		\begin{split}
			|u(x)|  \le \frac{C}{ D^{2}(x)} \nrm{ \omg}_{L^1_* },
		\end{split}
	\end{equation} and its second component further satisfies the estimate \begin{equation}\label{eq:u2-decay2}
		\begin{split}
			\frac{|u_{2}(x)|}{x_{2}} \le \frac{C}{D^{3}(x)}\min\left\{  {\nrm{ \omg}_{L^1_* }} ,D(x) \nrm{\omg}_{L^{1}} \right\} . 
		\end{split}
	\end{equation} 
\end{lemma}\begin{proof}
	We write 
	\begin{equation}\label{eq_u1}
		\begin{split}
			u_{1}(x)  & = - \frac{1}{2\pi} \int_{\bbR^2_+} \left( \frac{x_2-y_2}{|x-y|^2} - \frac{x_2+y_2}{|x-\bar{y}|^2} \right) \omg(y) dy \\
			& = -\frac{1}{2\pi} \int_{\bbR^2_+} \frac{4x_2^2}{|x-y|^2|x-\bar{y}|^2} y_2\omg(y) dy + \frac{1}{2\pi} \int_{\bbR^2_+} \left( \frac{1}{|x-y|^2} + \frac{1}{|x-\bar{y}|^2}  \right) y_2\omg(y) dy  \\
			& = \frac{1}{2\pi} \int_{\bbR^2_+} \frac{2(x_1-y_1)^{2} + 2(y_2^2-x_2^2)}{|x-y|^2|x-\bar{y}|^2} y_2\omg(y) dy
		\end{split}
	\end{equation} Similarly, \begin{equation}
		\label{eq_u2_temp}
		\begin{split}
			u_{2}(x) & =  \frac{1}{2\pi} \int_{\bbR^2_+} \left( \frac{x_1-y_1}{|x-y|^2} - \frac{x_1-y_1}{|x-\bar{y}|^2} \right) \omg(y) dy   =\frac{1}{2\pi} \int_{\bbR^2_+} \frac{4(x_1-y_1)x_2}{|x-y|^2|x-\bar{y}|^2} y_2\omg(y)dy. 
		\end{split}
	\end{equation} Therefore, \eqref{eq:K-decay} follows using that $|x-\bar{y}| \ge \max\{x_{2}+ y_{2},
	|x-y|\}$. From \eqref{eq:K-decay}, the bound \eqref{eq:u-decay} readily follows. \eqref{eq:u2-decay2} directly follows from \eqref{eq_u2_temp}.
\end{proof}

\begin{lemma}\label{lem:vel-L-infty}
	We have \begin{equation}\label{eq:vel-L-infty-log}
		\begin{split}
			\nrm{u}_{L^{\infty}} \le C\nrm{\omg}_{L^{2}}\left( 1 + \log_{+}\left( \frac{ \nrm{\omg}^{2}_{L^{\infty}} \nrm{ \omg}_{L^1_* } }{\nrm{\omg}_{L^{2}}^{3} } \right) \right).
		\end{split}
	\end{equation}
	As a consequence, for any $0<\alp<1$, there exists a constant $C_{\alp}>0$ such that \begin{equation}\label{eq:vel-L-infty-alpha}
		\begin{split}
			\nrm{u}_{L^{\infty}} \le C_{\alp}\nrm{\omg}_{L^{2}}^{1-\alp}\left( \nrm{\omg}_{L^{2}}^{\alp} + \nrm{\omg}_{L^{\infty}}^{2\alp/3} \nrm{ \omg}_{L^1_* }^{\alp/3} \right).
		\end{split}
	\end{equation}
\end{lemma}
\begin{remark}\label{rmk:stream-Hardy-bound}
	These upper bounds hold also for $|\psi(x)/x_{2}|$, since \begin{equation*}
		\begin{split}
			\frac{\psi(x)}{x_{2}} = -\frac{1}{x_2}\int_0^{x_2} u_{1}(x_1,z) dz . 
		\end{split}
	\end{equation*}
\end{remark}
\begin{proof}
	Let us apply \eqref{eq:K-decay} to the integral \eqref{eq:K-def} for the velocity, and split the integration domain into regions (i) $\{ |x-y| < l \}$, (ii) $\{ l \le |x-y| < L \}$, and (iii) $\{  L \le |x-y|  \}$ where  \begin{equation*}
		\begin{split}
			l := \frac{\nrm{\omg}_{L^{2}} }{ \nrm{\omg}_{L^\infty} }, \qquad L^{2}:= \frac{\nrm{ \omg}_{L^1_* }}{\nrm{\omg}_{L^{2}}}.  
		\end{split}
	\end{equation*} When $l \ge L$, we do not need to estimate the integral in the intermediate region (ii), which corresponds precisely to the case when the argument of $\log_{+}$ is less than 1 in \eqref{eq:vel-L-infty-log}. 
	
	For the regions (i), (ii), and (iii), we may bound the integrals by \begin{equation*}
		\begin{split}
			Cl\nrm{\omg}_{L^\infty}, \quad C\nrm{\omg}_{L^2}\log_+\left( \frac{L}{l} \right), \quad CL^{-2}\nrm{ \omg}_{L^1_* }, 
		\end{split}
	\end{equation*} respectively, where the bound for (iii) is a direct consequence of Lemma \ref{lem:vel-decay}. Plugging the definitions of $L$ and $l$ into above and taking the sum finishes the proof. 
\end{proof}

\subsubsection{Kinetic energy}\label{subsubsec:energy} For $\omg$ defined on $\bbR^2_+$, its kinetic energy $E=E[\omega]$ on $\bbR^2_+$ is defined by \begin{equation}\label{eq:E-def}
	\begin{split}
		E[\omega]=  -\frac12 \int_{\mathbb{R}^2_+}\psi(x)\omega(x)\,dx= \frac12 \int_{\mathbb{R}^2_+}|u(x)|^{2} \,dx.
	\end{split}
\end{equation}
	Alternatively, the energy can be expressed in terms of the double integral \begin{equation*}\label{eq:E-double-integral}
		\begin{split}
			E[\omg] = \frac12 \iint_{\bbR^2_+ \times \bbR^2_+} \omg(x)\bfG(x,y) \omg(y) dxdy 
		\end{split}
	\end{equation*}  where $\bfG$ was defined in \eqref{eq:G-def}. We also define the \emph{interaction energy} between two vorticities as follows:
	\begin{equation}\label{eq:E-int-def}
		\begin{split}
			E_{inter}[\omg,\tld{\omg}] := \frac12 \iint_{\bbR^2_+ \times \bbR^2_+} \omg(x) \bfG(x,y) \tld{\omg}(y) dxdy = E_{inter}[\tld\omg,{\omg}].
		\end{split}
	\end{equation}
	Since $\bfG(x,y) \geq 0$ for all $x, y \in \bbR^2_{+}$, we have
	$E[\omega]\ge 0$ for any $\omg\geq 0$, as long as it is finite. Furthermore, when $\omg, \tld\omg \ge 0$ on $\mathbb{R}^2_+$, we have $E_{inter}[\omg,\tld{\omg}] \ge 0$.
	
	Below we prove various bounds on the interaction energy. These inequalities are inspired by \cite[Proposition 2.1]{AC2019}, but we do not require the vorticity to be bounded in $L^{1}$. 
	\begin{proposition}\label{prop:Energy-X}
		Let $\omg, \tld{\omg} \in L^{2} \cap L^1_* (\bbR^2_+)$ and $\psi, \tld{\psi}$ be their corresponding stream functions. We have \begin{equation}\label{eq:E-inter-unsymm}
			\begin{split}
				\left|E_{inter}[\omg,\tld{\omg}]\right| \le \frac12 \min\left\{ \nrm{x_{2}^{-1}\tld{\psi}}_{L^{\infty}(\supp(\omg))}\nrm{\omg}_{L^1_* (\bbR^2_+)}  ,  \nrm{x_{2}^{-1}{\psi}}_{L^{\infty}(\supp(\tld\omg))}\nrm{\tld\omg}_{L^1_* (\bbR^2_+)}   \right\}, 
			\end{split}
		\end{equation} \begin{equation}\label{eq:Energy-interaction-X}
			\begin{split}
				\left|E_{inter}[\omg,\tld\omg]\right| \le C(\nrm{ \omg}_{L^1_* (\bbR^2_+)} \nrm{\omg}_{L^{2}(\bbR^2_+)}\nrm{\tld\omg}_{L^1_* (\bbR^2_+)} \nrm{\tld\omg}_{L^{2}(\bbR^2_+)})^{1/2},
			\end{split}
		\end{equation}and\begin{equation}\label{eq:Energy-difference-X}
			\begin{split} |E[\omg] - E[\tld\omg]| \le C(\nrm{\omg-\tld\omg}_{L^1_* (\bbR^2_+)} \nrm{\omg-\tld\omg}_{L^{2}(\bbR^2_+)}\nrm{\omg+\tld\omg}_{L^1_* (\bbR^2_+)} \nrm{\omg+\tld\omg}_{L^{2}(\bbR^2_+)})^{1/2}.\end{split}		\end{equation}
		If there is a positive distance $D$ between $\supp(\omg)$ and $\supp(\tld\omg)$ in $\bbR^{2}_{+}$, then we have  \begin{equation}\label{eq:Energy-interaction-support-separation}
			\begin{split}
				E_{inter}[\omg,\tld{\omg}] \le \frac{C}{D^{2}}\nrm{ \omg}_{L^1_* }\nrm{\tld\omg}_{L^1_* }. 
			\end{split}
		\end{equation}
	\end{proposition}
	\begin{proof}
		To begin with, to obtain \eqref{eq:E-inter-unsymm}, we just note that \begin{equation*}
			\begin{split}
				E_{inter}[\omg,\tld\omg] =  \frac12\int_{ \supp(\omg) }  (-\tld\psi)  \omg \, dx =  \frac12\int_{ \supp(\omg) } (-x_{2}^{-1}\tld\psi) \, (x_{2}\omg) \, dx.  
			\end{split}
		\end{equation*}
		
		Let us now prove \eqref{eq:Energy-interaction-X}: cut the integral into regions $\{ |x-y| \ge L \}$ and $\{ |x-y| < L \}$ for some $L>0$ to be determined. In the former, use the linear bound on the log to get \begin{equation*}
			\begin{split}
				\iint_{  \{ |x-y| \ge L \} } \omg(x) \bfG(x,y) \tld{\omg}(y) dxdy \le \iint_{  \{ |x-y| \ge L \} } \omg(x) \frac{Cx_{2}y_{2}}{|x-y|^{2}} \tld{\omg}(y) dxdy \le CL^{-2}\nrm{ \omg}_{L^1_* }\nrm{\tld\omg}_{L^1_* }. 
			\end{split}
		\end{equation*} To bound the latter, we first observe $\nrm{ x_{2}^{1/3} \omg }_{L^{3/2}} \le \nrm{\omg}_{L^1_* }^{1/3} \nrm{\omg}_{L^{2}}^{2/3}. $ Then, we apply \eqref{eq_alpha} with $\alpha=\frac13$ to bound \begin{equation*}
			\begin{split}
				\bfG(x,y)\mathbf{1}_{ \{ |x-y| < L \} } \le \frac{C(x_2y_2)^{1/3}}{|x-y|^{2/3}}\mathbf{1}_{ \{ |x-y| < L \} } =: (x_2y_2)^{1/3} \tilde{\bfG}_{L}(x-y)
			\end{split}
		\end{equation*} and use Young's convolution inequality to obtain \begin{equation*}
			\begin{split}
				\left| \iint x_{2}^{1/3}\omg(x)\tld{\bfG}_{L}(x-y)y_{2}^{1/3} \tld\omg(y) dxdy \right| \le \nrm{\tld	\bfG_{L}}_{L^{3/2}}\nrm{ x_{2}^{1/3}\omg(x)}_{L^{3/2}} \nrm{y_{2}^{1/3} \tld\omg(y)   }_{L^{3/2}}
			\end{split}
		\end{equation*} with $\nrm{\tld	\bfG_{L}}_{L^{3/2}} = CL^{2/3}$ to derive the bound \begin{equation*}
			\begin{split}
				\left|\iint_{  \{ |x-y| < L \} } \omg(x) \bfG(x,y) \tld{\omg}(y) dxdy \right|\le CL^{2/3}\nrm{ \omg}_{L^1_* }^{1/3}\nrm{\omg}_{L^{2}}^{2/3}  \nrm{\tld\omg}_{L^1_* }^{1/3} \nrm{\tld\omg}_{L^{2}}^{2/3} .
			\end{split}
		\end{equation*} Optimizing in $L$ gives the claimed bound.  The proof of \eqref{eq:Energy-difference-X} directly follows from \eqref{eq:Energy-interaction-X}, by noting that $E[\omg] - E[\tld\omg] = E_{inter}[\omg- \tld\omg, \omg+ \tld\omg]$.
		
		For the last statement, one may just bound \begin{equation*}
			\begin{split}
				\log\left( 1 + \frac{4 x_2 y_2}{|x-y|^{2}} \right) \le \frac{Cx_2 y_2}{|x-y|^{2}} \le \frac{C}{D^{2}}
			\end{split}
		\end{equation*} for $x\in \supp(\omg)$ and $y\in \supp(\tld\omg)$.
	\end{proof}  
	
	Simply taking $\omg = \tld\omg$ gives the following energy inequality. 
	\begin{corollary}[Energy inequality]  \label{cor:energy}
		Let $\omg \in L^{2} \cap L^1_* (\bbR^2_+)$. Then, its kinetic energy is bounded and satisfies  \begin{equation}\label{eq:Energy-X}
			\begin{split}
				E[\omg] \le C\nrm{ \omg}_{L^1_* (\bbR^2_+)} \nrm{\omg}_{L^{2}(\bbR^2_+)} 
			\end{split}
		\end{equation} with a universal constant $C>0$. 
	\end{corollary}
	
	\begin{remark}
		A similar computation gives a kinetic energy bound for $L^{p} \cap L^1_* $ vorticity for $p > 1$: \begin{equation*}
			\begin{split}
				E[\omg] \le C(p) \nrm{ \omg}_{L^1_* (\bbR^2_+)}^{\tht} \nrm{\omg}_{L^{p}(\bbR^2_+)}^{2-\tht}, \qquad \tht = \frac{4(p-1)}{3p-2}.
			\end{split}
		\end{equation*} This inequality should be useful in extending Theorem \ref{thm:N-Lamb-dipoles} to the case of energy maximizers with the $L^{p} \cap L^1_* $ constraint. 
	\end{remark}

	\subsection{Lamb dipoles}\label{subsec:Lamb}
	
	The Lamb dipole was  introduced  by 
	H. Lamb \cite[p. 231]{Lamb} and independently by S. A. Chaplygin \cite{Chap1903}. We have fixed the  {normalized Lamb dipole} to be 
	\begin{equation*} 
		\bar\omega_{Lamb}(x)=  \frac{-2c_L}{    J_0(c_L)} \, J_1(c_L r)  \mathbf{1}_{[0,1]}(r) \sin(\tht) 
	\end{equation*}  in the polar coordinate system $(r,\tht)$.  The function $ \frac{-2c_L}{    J_0(c_L)} \, J_1(c_L r) $ is strictly positive for $r < 1$, and in particular $\bar{\omg}_{Lamb} \ge 0$ on $\mathbb{R}^2_{+}$. Furthermore, it vanishes when $r = 1$, and therefore $\bar\omega_{Lamb}$ is supported in $|x| \le 1$ and Lipschitz continuous. 
	
	The stream function $\bar{\psi}_{Lamb}$ of the Lamb dipole can be explicitly computed, by writing it in the form $\lap (\bar{\psi}_{Lamb}) = \lap (f(r)\sin(\tht)) = \bar{\omg}_{Lamb}$ and solving the ODE in $r$ for $f(r)$. We have \begin{equation}\label{eq:stream-vorticity}
		\begin{split}
			\bar{\psi}_{Lamb}(x) = -r\sin(\tht) + \frac{2 J_{1}(c_{L}r) }{c_{L} J_{0}(c_{L})} \sin(\tht) , \qquad r \le 1 
		\end{split}
	\end{equation} and we can observe the following \textit{stream-vorticity relation} \begin{equation}\label{eq:stream-vorticity-rel}
		\begin{split}
			\bar\omega_{Lamb}=(c_L)^2 (-\bar\psi_{Lamb} -  x_2)_{+}
		\end{split}
	\end{equation} where $(f)_+ := \max\{ f , 0 \}$. This relation implies that $\bar\omega_{Lamb}$ is a traveling wave solution to \eqref{eq:2D-Euler} with velocity $\mathbf{e}_{1}$: $\bar\omega_{Lamb}(x - t\mathbf{e}_1)$ solves \eqref{eq:2D-Euler}. 
	
	We can explicitly compute the enstrophy, impulse, and energy of $\bar\omega_{Lamb}$. To begin with, the enstrophy is given by \begin{equation*}
		\begin{split}
			K[\bar\omega_{Lamb}] =\int_{\mathbb{R}^2_+} \bar\omega_{Lamb}^{2}(x) \,dx = \int_0^{\pi}\int_0^1  \left(\frac{-2c_L}{    J_0(c_L)} \, J_1(c_L r) \right)^{2} \sin^2(\tht) \, rdr d\tht  = \pi c_{L}^{2}. 
		\end{split}
	\end{equation*} The impulse is  $$\mu[\bar\omega_{Lamb}] =\int_{\mathbb{R}^2_+}x_2\bar\omega_{Lamb}(x) \,dx = \int_0^{\pi}\int_0^1 r\sin(\tht) \, \frac{-2c_L}{    J_0(c_L)} \, J_1(c_L r)  \sin(\tht) \, r dr d\tht =\pi.$$ Lastly, the energy can be computed using the {stream-vorticity relation \eqref{eq:stream-vorticity-rel}:} \begin{equation*}\begin{split}
			E[\bar\omega_{Lamb}] &=-\frac12\int_{\bbR^2_+}\bar\psi_{Lamb}(x)\bar\omg_{Lamb}(x)dx \\
			&=\frac12\int_{\bbR^2_+}(-\bar\psi_{Lamb}(x)-x_2)\bar\omg_{Lamb}(x)dx+\frac12\int_{\bbR^2_+}x_2\bar\omg_{Lamb}(x)dx =\pi.
		\end{split}
	\end{equation*}  
	
	Combining the main theorem of Abe--Choi  \cite[Theorems 1.3 and 1.5]{AC2019} with Corollary \ref{cor:energy} gives the following \textbf{sharp energy inequality}. 
	
	\begin{theorem}\label{thm:AC2019}
		For any $\kpp, \mu > 0$, the rescaled Lamb dipole $\omega_{Lamb}^{\kpp,\mu}$ defined in \eqref{eq:lamb-rescale-def} is the unique maximizer, up to shifts in the $x_{1}$-coordinate, of the energy in the class $\widetilde{\mathcal{A}}_{\kappa,\mu}$ defined in \eqref{eq:admissible}.  Furthermore, we have the following sharp energy inequality for $\omg \ge 0$ belonging to $L^{2} \cap L^1_* $: \begin{equation}\label{eq:energy-bound-sharp}
			\begin{split}
				E[\omg] \le C_{L} \nrm{\omg}_{L^1_* (\bbR^2_+)} \nrm{\omg}_{L^{2}(\bbR^2_+)} , \qquad C_{L} := \frac{1}{ \sqrt{\pi} c_{L}}. 
			\end{split}
		\end{equation} and the equality holds if and only if $\omg$ is an $x_{1}$-translate of $\omg_{Lamb}^{\kpp,\mu}$ for some $\kpp, \mu$.
	\end{theorem}
	
	\begin{proof}
		The first statement follows from Abe--Choi  \cite[Theorems 1.3 and 1.5]{AC2019}; see also \cite{CJY,ACJ}. Therefore, it suffices to prove \eqref{eq:energy-bound-sharp}, but this is an immediate consequence of this variational principle and Corollary \ref{cor:energy}, simply because the constant $C>0$ in the corollary {is} achieved by plugging in a Lamb dipole into the inequality. 
	\end{proof}

	\section{Decompositions and Bootstrap Assumptions}\label{sec:overview}
	
	\subsection{Assumptions and notations}\label{subsec:assump-nota} For convenience, we recall the assumptions of Theorem \ref{thm:N-Lamb-dipoles} and the notations that we shall use in the following. To begin with, we fix some $N \ge 2$, and recall that the $N$-Lamb dipole function was defined in \eqref{eq:N-dipole-def} by \begin{equation*}
		\begin{split}
			\overline{\omg}^{\boldsymbol{\bar\kpp}, \boldsymbol{\bar\mu}}_{Lamb, \mathbf{p}}(x) = \sum_{i=1}^{N} \omg_{Lamb}^{\bar\kpp_{i},\bar\mu_{i}}(x - p_{i} \mathbf{e}_{1} ).
		\end{split}
	\end{equation*}We had the following assumptions in Theorem \ref{thm:N-Lamb-dipoles}. \begin{itemize}
		\item The $L^{2}$-norms are strictly ordered: $\bar\kpp_{1} > \bar\kpp_{2} > \cdots > \bar\kpp_{N}.$ This means that the corresponding velocities of Lamb dipoles are also strictly ordered: $\bar{V}_{1} > \cdots > \bar{V}_{N}$, since the velocity is directly proportional to the $L^{2}$-norm, given by $\bar V_i =  C_{L}\bar\kpp_{i}=  \bar\kpp_{i}/(\sqrt{\pi}c_{L})$.
		\item The initial locations are sufficiently separated: the components of $\bar\bfp$ satisfy $\bar{p}_{i} > \bar{p}_{i+1} + D_{0}$ for all $1 \le i \le N - 1$. 
		\item The initial data $\omg_{0}$ is $\dlt_{0}$-close to $\overline{\omg}^{\boldsymbol{\bar\kpp}, \boldsymbol{\bar\mu}}_{Lamb, \mathbf{p}(0)}$ in the norm $L^{2} \cap L^1_* $, and bounded by $M$ in $L^{1}\cap L^{\infty}$. Furthermore, the measure of the support of $\omg_{0}$ is bounded by $M$. \end{itemize}
	In what follows, we let the constant $C>0$ depend on  $N, M, \boldsymbol{\bar{\kpp}}, \boldsymbol{\bar{\mu}}$, but never on the time variable $t$, or $\varepsilon, \delta_0, D_0$.  Whenever it becomes necessary, we shall take $\varep_{0}, \dlt$ smaller  and $D_{0}$ larger, but only in a way depending on $N, M, \boldsymbol{\bar{\kpp}}, \boldsymbol{\bar{\mu}}$. Here, $\varep$ is the maximal distance between $\omg$ and the $N$-Lamb dipole for all $t\ge0$, see \eqref{eq:varep-close}, with $0<\varep<\varep_{0}$, and $\dlt>0$ is the smallness parameter introduced for convenience in our bootstrap assumptions, see \eqref{eq:B1}--\eqref{eq:B3} below. Lastly, we shall take $\dlt_{0} = \dlt/C_{0}$, for some large constant $C_{0} = C_{0}(N, M, \boldsymbol{\bar{\kpp}}, \boldsymbol{\bar{\mu}}) \ge 100N$ determined through the bootstrap procedure. Here is a dependency diagram:  \begin{equation*}\label{eq:params-diagram}
		\begin{split}
			N, M, \boldsymbol{\bar{\kpp}}, \boldsymbol{\bar{\mu}} \longrightarrow \varep_{0} \longrightarrow \dlt \longrightarrow D_{0} \longrightarrow \dlt_{0}. 
		\end{split}
	\end{equation*}
	
	\subsection{Decomposition into $N$ vortices}\label{subsec:decomp-N-vortices}
	We define the time-dependent ``borders'' \begin{equation}\label{eq:L-def}
		\begin{split}
			L_{i}(t) := \frac12\left( \bar{p}_{i} + \bar{p}_{i+1} + (\bar{V}_{i} + \bar{V}_{i+1})t \right), \qquad 1 \le i \le N-1. 
		\end{split}
	\end{equation} This is simply the midpoint of the centers of the Lamb dipoles starting from $\bar{p}_{i}, \bar{p}_{i+1}$ and moving with speed $\bar{V}_{i}$ and $\bar{V}_{i+1}$, respectively.  It will be convenient to also set $L_{0} := +\infty$ and $L_{N} := -\infty$, and we decompose the solution by \begin{equation}\label{eq:vort-decomp-def}
		\begin{split}
			\omg(t,\cdot) = \sum_{i=1}^{N} \omg_{i}(t,\cdot), \qquad \omg_{i}(t,\cdot) := \omg(t,\cdot) \mathbf{1}_{ \{ L_{i}(t) \le x_{1} < L_{i-1}(t)    \} }(\cdot). 
		\end{split}
	\end{equation}  
	See Figure~\ref{fig_borders} for an illustration of $L_i(t)$ and $\omega_i(t,\cdot)$ in the case of $N=3$.
	
	\begin{figure}[htbp]
		\includegraphics[scale=1]{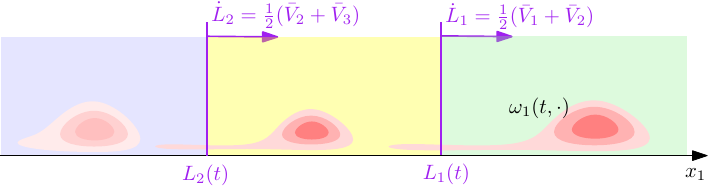} 
		\caption{\label{fig_borders} An illustration of the borders $L_i$ and the vorticities $\omega_i(t,\cdot)$ in the case of $N=3$. Here $\omega_1,\omega_2,\omega_3$ are the vorticities lying on the green, yellow and blue backgrounds respectively.} 
	\end{figure}

	The  velocity and stream function corresponding to $\omg_{i}$ are denoted by $u_{i}$ and $\psi_{i}$, respectively. It will be convenient to introduce the partial sums:   \begin{equation}\label{eq:vort-decomp-def2}
		\begin{split}
			\omg_{\ge i}(t,\cdot) := \sum_{i \le j \le N} \omg_{j}(t,\cdot) = \omg(t,\cdot) \mathbf{1}_{ \{  x_{1} < L_{i-1}(t)    \} }(\cdot), \qquad \omg_{\le i}(t,\cdot) := \sum_{1 \le j \le i} \omg_{j}(t,\cdot) = \omg(t,\cdot) \mathbf{1}_{ \{ L_{i}(t) \le  x_{1} \} }(\cdot)
		\end{split}
	\end{equation} and denote $u_{\ge i}, u_{\le i}$ and $\psi_{\ge i}, \psi_{\le i}$ as the corresponding sums of $u_{j}$ and $\psi_{j}$, respectively.  For simplicity, we set \begin{equation}\label{eq:vort-decomp-params}
		\begin{split}
			\kpp_{i}(t) := \nrm{\omg_{i}(t,\cdot)}_{L^{2}}, \quad K_{i}(t) := \nrm{\omg_{i}(t,\cdot)}_{L^{2}}^{2}, \quad \mu_{i}(t) := \nrm{\omg_{i}(t,\cdot)}_{L^1_* },  \quad E_{i}(t) := E[\omg_{i}(t,\cdot)]. 
		\end{split}
	\end{equation} Note the additivity of enstrophy and impulse:  \begin{equation}\label{eq:vort-decomp-params2}
		\begin{split}
			K_{\le i}(t) := \sum_{1 \le j \le i} K_{j}(t) = \nrm{ \omg_{ \le i }(t,\cdot) }_{L^{2}}^{2}, \qquad \mu_{\le i}(t) := \sum_{1 \le j \le i} \mu_{j}(t) = \nrm{ \omg_{ \le i }(t,\cdot) }_{L^1_* }. 
		\end{split}
	\end{equation} Similar formulae hold for $K_{\ge i}$ and $\mu_{\ge i}$, defined by the corresponding sums of $K_{j}$ and $\mu_{j}$ for $i \le j \le N$. In the subsequent sections, we shall frequently use the following computation.
	\begin{lemma}[Initial closeness] \label{lem:ini-close}
		At the initial time, we have \begin{equation}\label{eq:ini-smallness}
			\begin{split}
				\sum_{i=1}^{N} \left(|K_{i}(0) - \bar{K}_{i}| + |\kpp_{i}(0) - \bar{\kpp}_{i}| + |\mu_{i}(0) - \bar{\mu}_{i}| + |E_{i}(0) - C_{L} \bar{\kpp}_{i}\bar{\mu}_{i}| \right) \le \frac{C}{C_{0}} \dlt 
			\end{split}
		\end{equation} for any $C_{0}>0$ by taking $\dlt_{0} := \dlt/C_{0}$ for some $C>0$ depending on $M, N, \boldsymbol{\bar{\kpp}}, \boldsymbol{\bar{\mu}}$. 
	\end{lemma}
	\begin{proof}
		This is a direct consequence of the initial data assumptions from Theorem \ref{thm:N-Lamb-dipoles}. 
	\end{proof}
	
	\subsection{Bootstrap assumptions} For $t\ge0$, we take the following bootstrap assumptions \begin{equation}\label{eq:B1}\tag{B1}
		\begin{split}
			K_{ \le i }(t) \le   K_{\le i}(0) + \dlt , \qquad 1 \le i \le N-1, 
		\end{split}
	\end{equation} \begin{equation}\label{eq:B2}\tag{B2}
		\begin{split}
			\mu_{ \le i }(t) \le   \mu_{\le i }(0) + \dlt, \qquad 1 \le i \le N-1, 
		\end{split}
	\end{equation} and \begin{equation}\label{eq:B3}\tag{B3}
		\begin{split}
			E_{i}(t) \ge E_{i}(0) - \dlt , \qquad 1 \le i \le N. 
		\end{split}
	\end{equation}
	Note that the lower bound assumption on the energy is imposed for all $1 \le i \le N$; {unlike the enstrophy and impulse, the energy has interaction terms, so information on $N-1$ of them does not directly give bounds on the remaining one.}
	
	\subsection{More refined decompositions of the vortices}
	
	In the following, we will introduce two different ways to decompose $\omega(t,\cdot)$ further. The decompositions are based on two different viewpoints, and as we will see, the interplay between these two viewpoints will be crucial in the proof. 
	
	The first way takes an Eulerian viewpoint by decomposing each $\omega_i(t,\cdot)$ into the parts that are more ``central'' in the interval $[L_i(t), L_{i-1}(t)]$ and the two ``error terms'' that are close to the left border $L_i(t)$ and right border $L_{i-1}(t)$.  The second way takes a Lagrangian viewpoint, by decomposing $\omega_{\geq i}(t)$ into the ``gained'' vorticity that originally comes from the right of the border $L_{i-1}(0)$ at $t=0$, and the rest that originally comes from the left of $L_{i-1}(0)$. See Figure~\ref{fig_euler} and \ref{fig_lagrange} for  illustrations of the two decompositions.

	\subsubsection{Eulerian decomposition: left, center, and right} \label{sec_euler} We  set additional borderlines \begin{equation}\label{eq:Lpm-def}
		\begin{split}
			L^{\pm}_{i}(t) := L_{i}(t) \pm \frac1{10}\left( D_{0} + (\bar{V}_{i} - \bar{V}_{i+1})t \right), \qquad 1 \le i \le N-1. 
		\end{split}
	\end{equation} Then, we introduce the left and right ``error terms'' \begin{equation}\label{eq:omg-err-def}
		\begin{split}
			\omg_{i,err}^{(\ell)}(t,\cdot) := \omg_{i}(t,\cdot) \mathbf{1}_{ \{  L_{i}(t) \le x_{1} < L_{i}^{+}(t) \} }(\cdot), \quad \omg_{i,err}^{(r)}(t,\cdot) := \omg_{i}(t,\cdot) \mathbf{1}_{ \{  L_{i-1}^{-}(t) \le x_{1} < L_{i-1}(t) \} }(\cdot),
		\end{split}
	\end{equation} for all $1 \le i \le N$, where we set $\omg_{1,err}^{(r)} := 0$ and $\omg_{N,err}^{(\ell)} := 0$ to unify notation. This defines the ``central'' part $\omg_{i,cen}$ by the decomposition \begin{equation}\label{eq:omg-cen-def}
		\begin{split}
			\omg_{i,cen}(t,\cdot) := \omg_{i}(t,\cdot) \mathbf{1}_{ \{  L_{i}^{+}(t) \le x_{1} < L_{i-1}^{-}(t) \}},
		\end{split}
	\end{equation} so that we have $\omg_{i}(t,\cdot) = \omg_{i,err}^{(\ell)}(t,\cdot) + \omg_{i,cen}(t,\cdot) + \omg_{i,err}^{(r)}(t,\cdot)$ for all $1 \le i \le N$; see Figure~\ref{fig_euler} for an illustration.

	\begin{figure}[htbp]
		\includegraphics[scale=1]{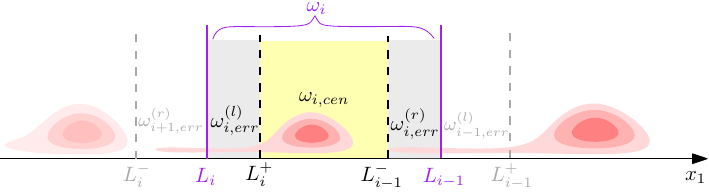} 
		\caption{\label{fig_euler} An illustration of the Eulerian decomposition of $\omega_i$ into the ``central part'' $\omega_{i,cen}$ (on yellow background) and the two ``error terms''  $\omega_{i,err}^{(\ell)}$, $\omega_{i,err}^{(r)}$ on the left and right (on gray background).} 
	\end{figure}

	\subsubsection{Lagrangian decomposition: gain and the rest} For $i \ge 2$, we decompose $\omg_{\ge i}$ by  \begin{equation}\label{eq:omg-gain}
		\begin{split}
			\omg_{\ge i}(t,\cdot) =: {\omg}_{\ge i, *}(t,\cdot) + \omg_{\ge i,gain}(t,\cdot)
		\end{split}
	\end{equation} where $\omg_{\ge i,gain}$ is defined by the part of $\omg_{\ge i}$ at time $t$ which was located to the right of {$L_{i-1}(0)$} at the initial time; whereas ${\omg}_{\ge i, *}$ was located to the left of {$L_{i-1}(0)$} at the initial time; see Figure~\ref{fig_lagrange} for an illustration. More precisely, we define \begin{equation}\label{eq:omg-gain-def}
		\begin{split}
			{\omg}_{\ge i, gain}(t,\cdot) := \omg(t,\cdot) \mathbf{1}_{ \{ x_{1} <{ L_{i-1}(t)} \} } \, \mathbf{1}_{ \{\Phi^{-1}_{1}(t,x) \ge {L_{i-1}(t)}\}},
		\end{split}
	\end{equation} where $\Phi^{-1}(t,x) = (\Phi^{-1}_1(t,x),\Phi^{-1}_2(t,x))$ is the inverse of the flow map $\Phi$ at time $t$; that is, $\Phi^{-1}(t,\Phi(t,x)) = x = \Phi(t,\Phi^{-1}(t,x))$.  
	Here, $\Phi$ is the flow map corresponding to $u(t,x) = \sum_{i=1}^{N}u_{i}(t,x)$; \begin{equation*}
		\begin{split}
			\frac{d}{dt} \Phi(t,x) = u(t,\Phi(t,x)), \qquad \Phi(0,x) = x. 
		\end{split}
	\end{equation*} 
	As a consequence of our bootstrap assumptions, we will show that fluid particles can only cross each $L_{i-1}(t)$ from right to left\footnote{This will be proved in Lemma~\ref{lem:u-border}.}, and therefore $\omg_{\ge i, gain}$ will be identified precisely as the part of $\omg_{ \ge i}$ which have crossed the line $L_{i-1}(t')$ to the left for some $0\le t' \le t$.

	\begin{figure}[htbp]
		\includegraphics[scale=1]{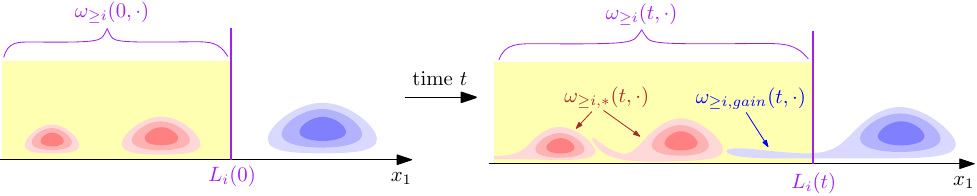} 
		\caption{\label{fig_lagrange} An illustration of the Lagrangian decomposition of $\omega_{\geq i}(t,\cdot)$ into the ``gained'' part $\omega_{\geq i, gain}$ and the rest $\omega_{\geq i,*}$. In this figure, we color the vorticity in red if the fluid particle is initially located on the left of {$L_{i-1}(0)$} at initial time; and blue otherwise. So $\omega_{\geq i,*}(t,\cdot)$ include all the red vorticity on yellow background; whereas $\omega_{\geq i,gain}(t,\cdot)$ include all the blue vorticity on yellow background.}  
	\end{figure}
	
	\section{Consequences of the Bootstrap Assumptions}\label{sec:boot-con}

	Throughout this section, we assume that the bootstrap assumptions \eqref{eq:B1}, \eqref{eq:B2}, and \eqref{eq:B3} hold for a time interval $[0,T]$ for some $T>0$. Then, for $t \in [0,T]$, we obtain various information on the solution, in the form of five lemmas. We emphasize again that the constant $C$ appearing in the statements does not depend on $T$.  In all of the lemmas, it is omitted that ``there exists $\varep_{0}>0$ depending on  $N, M, \boldsymbol{\bar{\kpp}}, \boldsymbol{\bar{\mu}}$ such that if we take $\dlt>0$ sufficiently small, $D_{0}$ sufficiently large depending on $\varep_{0}$, and $0<\varep<\varep_{0}$ then the following holds on $[0,T]$.''
	
	\subsection{Approximating by Lamb dipoles}\label{subsec:approx-Lamb-dipoles}
	
	In this section, under the bootstrap assumptions \eqref{eq:B1}--\eqref{eq:B3}, we show that each vortex $\omg_{i}$ is well-approximated by the corresponding initial Lamb dipole, and obtain an estimate on the shift. 
	
	\begin{lemma}\label{lem:approx-Lamb-dipoles}
		If the bootstrap assumptions \eqref{eq:B1}--\eqref{eq:B3} hold on $[0,T]$, then for all $t\in [0,T]$ we have \begin{equation}\label{eq:params-dlt-close}
			\begin{split}
				|K_{i}(t) - \bar{K}_{i}| + |\mu_{i}(t) - \bar{\mu}_{i}| + |E_{i}(t) - C_{L} \bar\kpp_{i} \bar\mu_{i} | \le C \dlt, 
			\end{split}
		\end{equation}\begin{equation}\label{eq:omg-1-Lamb-epsilon}
			\begin{split}
				\nrm{ \omg_{i}(t) - \bar\omg_{i,Lamb}(t) }_{L^{2} \cap  L^1_*  } \le \frac{\varepsilon}{N} , \qquad \text{where} \quad 	\bar\omg_{i,Lamb}(t,x):=  \omg_{Lamb}^{\bar\kpp_{i}, \bar\mu_{i}}(\cdot - p_{i}(t) \mathbf{e}_{1} )
			\end{split}
		\end{equation} for some $p_{i}(t)$ satisfying  \begin{equation}\label{eq:p1-B}
			\begin{split}
				|p_{i}(t)- \bar{p}_{i}  - \bar{V}_{i}\, t | \le C\varep^{1/2}(1+t) 
			\end{split}
		\end{equation} for all $1 \le i \le N$. 
	\end{lemma}
	
	\begin{proof} 
		In the proof, we first prove \eqref{eq:params-dlt-close} and \eqref{eq:omg-1-Lamb-epsilon} for $i=1,\dots,N$ using an inductive argument, starting with the $i=1$ case. We then prove the shift estimate \eqref{eq:p1-B} at the end.
		
		\smallskip
		\noindent \textbf{Case $i = 1$}. In this case, using \eqref{eq:B1}, \eqref{eq:B2}, and \eqref{eq:B3} for $i = 1$, and the initial data assumption with \eqref{eq:Energy-difference-X}, we obtain a lower bound on $E_{1}(t)$ \begin{equation*}
			\begin{split}
				E_{1}(t) \ge E_{1}(0) -\dlt > C_{L} \bar\kpp_{1} \bar\mu_{1} - C\dlt 
			\end{split}
		\end{equation*} as well as $\kpp_{1}(t) \le \bar{\kpp}_{1} + C\dlt$, $\mu_{1}(t) \le \bar{\mu}_{1} + C\dlt$. In particular, for any small $\varep_{1}>0$, by taking $\dlt>0$ smaller if necessary and appealing to Proposition \ref{prop:Lamb-energy-estimate}, we obtain \begin{equation}\label{eq:omg-1-Lamb-epsilon-temp}
			\begin{split}
				\nrm{ \omg_{i}(t) - \omg_{i,Lamb}(t) }_{L^{2} \cap  L^1_*  } \le  {\varepsilon_{1}} ,\quad 	 \omg_{i,Lamb}(t,x):=  \omg_{Lamb}^{\bar\kpp_{i}, \bar\mu_{i}}(\cdot - \tau_{i}(t) \mathbf{e}_{1} )
			\end{split}
		\end{equation} for $i=1$ with \textit{some} $\tau_{1}(t)$. Furthermore, from the general upper bound $E_{1}(t) \le C_{L} \kpp_{1}(t)\mu_{1}(t)$, we get \begin{equation*}
			\begin{split}
				C_{L} \kpp_{1}(t)\mu_{1}(t) > C_{L} \bar\kpp_{1} \bar\mu_{1} - C\dlt ,
			\end{split}
		\end{equation*} which combined with \eqref{eq:B1}$_{1}$ and \eqref{eq:B2}$_{1}$ again gives lower bounds  \begin{equation}\label{eq:K-mu-1-LB}
			\begin{split}
				K_{1}(t) > \bar{K}_{1} - C\dlt, \qquad \mu_{1}(t) > \bar\mu_{1} - C\dlt. 
			\end{split}
		\end{equation}
		This leads to \eqref{eq:params-dlt-close} for $i=1$.
		
		\smallskip
		\noindent \textbf{Case $2 \le i \le N-1$}. For $2 \le i \le N-1$, we use the following inductive argument starting with $i=2$. Applying \eqref{eq:params-dlt-close} for all $j=1,\dots,i-1$, we have \begin{equation}\label{eq:K-mu-j-LB}
			\begin{split}
				K_{j}(t) > \bar{K}_{j} - C\dlt, \qquad \mu_{j}(t) > \bar\mu_{j} - C\dlt \quad\text{ for all }j=1,\dots,i-1.
			\end{split}
		\end{equation} Combining \eqref{eq:K-mu-j-LB} with \eqref{eq:B1}$_{i}$ and \eqref{eq:B2}$_{i}$, we obtain \begin{equation}\label{eq:K-mu-UB}
			\begin{split}
				K_{i}(t) \le \bar{K}_{i} + C\dlt, \qquad \mu_{i}(t) \le \bar{\mu}_{i} + C\dlt. 
			\end{split}
		\end{equation} Together with this upper bound, using \eqref{eq:B3}$_{i}$ \begin{equation*}
			\begin{split}
				E_{i}(t) \ge E_{i}(0) - \dlt > C_{L} \bar{\kpp}_{i} \bar{\mu}_{i} - C\dlt 
			\end{split}
		\end{equation*} gives \eqref{eq:omg-1-Lamb-epsilon-temp} for $i$ with some $\tau_{i}(t)$. As in the case $i = 1$ above, with  \begin{equation*}
			\begin{split}
				C_{L}\bar\kpp_{i}\bar\mu_{i} - C\dlt < E_{i}(t) \le C_{L}\kpp_{i}(t) \mu_{i}(t), 
			\end{split}
		\end{equation*} we deduce the lower bounds \begin{equation*}\label{eq:muK1-LB}
			\begin{split}
				K_{i}(t) > \bar{K}_{i} - C\dlt, \qquad \mu_{i}(t) > \bar\mu_{i} - C\dlt ,
			\end{split}
		\end{equation*}
		leading to \eqref{eq:params-dlt-close}.
		
		\smallskip
		\noindent \textbf{Case $i = N$}. Again, in this case we may assume that \eqref{eq:K-mu-j-LB} is given for all $j < N$. Now using the conservation laws \begin{equation*}
			\begin{split}
				\sum_{i=1}^{N} \mu_{i}(t) = \sum_{i=1}^{N} \bar{\mu}_{i} +O(\dlt), \qquad 	\sum_{i=1}^{N} K_{i}(t) = \sum_{i=1}^{N} \bar{K}_{i} +O(\dlt)
			\end{split}
		\end{equation*} gives the upper bound \eqref{eq:K-mu-UB} for $N$. Proceeding similarly as above and using \eqref{eq:B3}$_{N}$, we obtain \eqref{eq:omg-1-Lamb-epsilon-temp} for $i = N$.
		
		\smallskip

		\noindent \textbf{Shift estimate}. 
		At this point, by \eqref{eq:omg-1-Lamb-epsilon-temp}, we know that for all $t \in [0,T]$, $\omg_{i}(t,\cdot)$ is ${\varep_{1}}$-close to the Lamb dipole $\omg_{Lamb}^{\bar\kpp_i, \bar\mu_i}$ after some shift $\tau_{i}(t)$, for each $1 \le i \le N$. Under this condition, we want to define $p_1(t)$ that satisfies \eqref{eq:omg-1-Lamb-epsilon} and \eqref{eq:p1-B}. We will provide this argument just for the case $i = 1$, the other cases being identical. 
		
		\medskip
		\noindent\underline{Step 1. Estimation on the oscillation of $\tau_1$.} Unfortunately, $\tau_{1}(t)$ defined by the property \eqref{eq:omg-1-Lamb-epsilon-temp} may not be unique, and is not necessarily continuous in $t$. Still, we will show that it cannot jump too much within any time interval of length $T_{1} := \min\{ T, \frac{R_{1}}{10 N \bar{V}_{1}} \}$, where $R_{1} := R(\bar\kpp_1, \bar\mu_1)$ is the support radius of the corresponding Lamb dipole defined in \eqref{eq:lamb-rescale-def}. Namely, for all $0\leq t < t' \leq T$, we will show
		\begin{equation}\label{eq:tau-jump}
			\begin{split}
				|\tau_{1}(t) - \tau_{1}(t')| \le R_{1}  \qquad \text{ if } 0<t'-t<T_1.
			\end{split}
		\end{equation} 
		To see this, we first note the global-in-time pointwise bound for the velocity $\nrm{u}_{L^{\infty}} \le 5N\bar{V}_{1}$, which holds even when one ``stacks up'' all the Lamb dipoles at a single point. (The maximum velocity of a Lamb dipole with traveling speed $\bar{V}_{1}$ is bounded by $4\bar{V}_{1}$.) For any $t' \in [t, t+T_1]$, assume towards a contradiction that $\tau_1(t') > \tau_{1}(t) + R_{1}$. Due to the global velocity bound, the fluid particles starting from $x_{1} \le \tau_{1}(t)$ at time $t$ (which contains almost half of the $L^{2}$ norm from $\omg_{1}(t,\cdot)$) lies to the left of $x_{1} = \tau_{1}(t) + 5N\bar{V}_{1} (t'-t) \le \tau_{1}(t) + R_{1}/2$ at time $t'$. Therefore, if $\tau_{1}(t') > \tau_{1}(t) + R_{1}$, the shifted Lamb dipole $\omg_{Lamb}^{\bar\kpp_{1}, \bar\mu_{1}}(\cdot - \tau_{1}(t') \mathbf{e}_{1} )$ is clearly going to miss all of them, which contradicts the defining property of $\tau(t')$ in \eqref{eq:omg-1-Lamb-epsilon-temp}. We can use a similar argument to rule out the possibility of $\tau_1(t') < \tau_{1}(t) - R_{1}$.
		
		\medskip
		\noindent
		\underline{Step 2. Definition of the displacement $p_1(t)$.} In the following, we will define $p_1(t)$ in an implicit way, by slightly 
		adjusting the argument in \cite[\S 3.2]{JYZ}.
		Let $g \in C^\infty(\mathbb{R})$ be a smooth test function that is odd, monotone increasing on $[-4R_{1},4R_{1}]$, and satisfies 
		\[g(x_{1}) = \begin{cases}
			x_1 & \text{ for } |x_{1}| \le 3R_{1}, \\
			4R_{1} & \text{ for } 4R_{1} \le x_{1} \le D_{0}/5, \\
			0 & \text{ for }x_{1} \ge D_{0}/4. 
		\end{cases}
		\]
		We then introduce the function $H(t,p)$ for $t\ge 0$ and $p \in \bbR$ as \begin{equation}\label{eq:def-H}
			\begin{split}
				H(t,p) := 	\int_{\bbR^2_+} \omg(t, x + p \, \bfe_1) g(x_1) dx ,
			\end{split}
		\end{equation} 
		which is differentiable in $t$ and $p$ since $g$ is smooth. 
		
		Using such function $H$, we will define $p_1(t)$ and show its properties through an induction argument, each time advancing by time $T_1$ until we hit time $T$. Namely, assuming at some $t_0 \in [0,T]$ we have 
		\begin{equation}\label{induction}
			\tau_1(t_0)-L_1(t_0)>\frac{D_0}{3},
		\end{equation} we will show the following holds for all $t\in [t_0,\min\{t_0+T_1,T\}]$:
		
		(a) There exists a unique $p_{1}(t) \in (\tau_1(t)-2R_1, \tau_1(t) + 2R_1)$ that satisfies \begin{equation}
			\begin{split}
				H(t,p_{1}(t)) = \int_{\bbR^2_+} \omg(t, x + p_{1}(t) \bfe_{1} ) g(x_{1}) dx = 0.\label{def_p}
			\end{split}
		\end{equation} 
		Furthermore, such $p_1(t)$ satisfies 
		\begin{equation}
			\label{temp_p_diff}
			|p_1(t)-\tau_1(t)|\leq C\varep_1,
		\end{equation}
		and 
		\begin{equation}\label{temp_goal_p}
			\nrm{ \omg_{1}(t) - \omg_{Lamb}^{\bar\kpp_{1}, \bar\mu_{1}}(\cdot - p_{1}(t) \mathbf{e}_{1} ) }_{L^{2} \cap  L^1_*  } \le C\varepsilon_{1}.
		\end{equation}
		
		(b) The function $p_1(t)$ defined above is differentiable in $t$ for $t\in[t_0, \min\{t_0+T_1, T\}]$, and satisfies
		\begin{equation}\label{eq:p1-B-goal}
			\begin{split}
				\Big|\frac{d}{dt}p_{1}(t)  - \bar{V}_{1} \Big| \le C\varep_{1}^{1/2}.
			\end{split}
		\end{equation} 
		
		\medskip
		Note that \eqref{induction} holds at $t_0=0$, since $|\tau_1(0)-\bar p_1| < R_1$  (where $\bar{p}_1$ is from the assumption of Theorem \ref{thm:N-Lamb-dipoles}), and $\bar p_1 - L_1(0)>\frac{D_0}{2}$ by definition of $L_1(0)$. Assuming \eqref{induction} holds at some $t_0\in[0,T]$, once we prove that (b) holds, it implies \[
		\frac{d}{dt} (p_1(t)-L_1(t)) \geq \frac{\bar V_1-\bar V_2}{2} - C\varep_1^{1/2} \quad \text{ for all }t\in [t_0,t_0+\min\{t_0+T_1,T\}].
		\]
		If $t_0+T_1 < T$, we integrate the above in $[t_0,t_0+T_1]$ and combine it with \eqref{temp_p_diff} to obtain
		\[
		\begin{split}
			\tau_1(t_0+T_1)-L_1(t_0+T_1) &\geq p_1(t_0+T_1)-L_1(t_0+T_1) - C\varep_1\\
			& \geq p_1(t_0)-L_1(t_0) + \frac{\bar V_1 - \bar V_2}{2} T_1 - C(\varep_1 + \varep_1^{1/2} T_1) \\
			&\geq  \tau_1(t_0)-L_1(t_0) + \frac{\bar V_1 - \bar V_2}{2} T_1 - C(2\varep_1 + \varep_1^{1/2} T_1) \\
			&\geq \tau_1(t_0)-L_1(t_0)\geq \frac{D_0}{3},
		\end{split}
		\]
		thus \eqref{induction} also holds at $t_0+T_1$. We can then do induction to conclude that (a) and (b) holds for all $t\in [0,T]$.
		
		\medskip
		\noindent\underline{Step 3. Existence and uniqueness of $p_1(t)$.} Let us first justify the induction step (a). For any $t\in [t_0, \min\{t_0+T_1, T\}]$, combining \eqref{induction} with \eqref{eq:tau-jump}, we know $\tau_1(t) - L_1(t) > D_0/4$. 
		Using the decomposition $\tld{\omg}_{i}(t) := \omg_{i}(t) - \omg_{i,Lamb}(t)$ and \eqref{eq:omg-1-Lamb-epsilon-temp}, we have \begin{equation*}
			\begin{split}
				H(t,\tau_1(t))& =\int_{\bbR^2_+} \omg_{1}(t, x + \tau_{1}(t) \bfe_1) g(x_1) dx  + \int_{\bbR^2_+} \omg_{\geq 2}(t, x + \tau_{1}(t) \bfe_1) g(x_1) dx \\
				&= \int_{\bbR^2_+} \omg_{1}(t, x + \tau_{1}(t) \bfe_1) g(x_1) dx \\
				& =  \int_{\bbR^2_+} \omg_{Lamb}^{\bar\kpp_1,\bar\mu_1}(x) g(x_1) dx + \int_{\bbR^2_+} \tld\omg_{1}(t, x + \tau_{1}(t) \bfe_1) g(x_1) dx.
			\end{split}
		\end{equation*}
		Here the second equality follows from the fact that $\omega_{\geq 2}(t,\cdot+\tau_1(t)\mathbf{e}_1)$ are supported on $\{x_1<L_1(t)-\tau_1(t)\}$, whereas $g$ is supported on $\{x_1>-D_0/4\}$; these supports are disjoint since $\tau_1(t) - L_1(t) > D_0/4$. For the two integrals on the right hand side, note that the first integral is zero by odd symmetry of $g$ in $x_{1}$, and the second integral is bounded by $C\varep_{1}$ using $\nrm{\tld{\omg}_{1}}_{L^2} \lesssim \varep_{1}$, with $C>0$ depending on the size of the support of $\tld{\omg}_1$. The above discussion yields \begin{equation}\label{eq:H-bound-tau}
			\begin{split}
				\left| H(t,\tau_{1}(t)) \right| \le C\varep_{1}. 
			\end{split}
		\end{equation} Now we compute, for points $(t,p)$ satisfying $|p - \tau_{1}(t)| \le 2 R_{1}$,
		\begin{equation*}
			\begin{split}
				(\rd_{p}H)(t,p) &= - \int_{ \bbR_+^2 } \omg_{1}(t, x+p\mathbf{e}_1) g'(x_{1}) dx \\
				&= 
				- \int_{ \bbR_+^2 }  \omega_{Lamb}^{\bar \kappa_1,\bar \mu_1}(x+(p-\tau_1(t))\mathbf{e}_1) dx - \int_{ \bbR_+^2 } \tld{\omega}_1(t,x+p\mathbf{e}_1) g'(x_1) dx\\
				&= - \underbrace{ \int_{ \bbR_+^2 }  \omega_{Lamb}^{\bar \kappa_1,\bar \mu_1}(x) dx}_{=: m_1} - \int_{ \bbR_+^2 } \tld{\omega}_1(t,x+p\mathbf{e}_1) g'(x_1) dx.
			\end{split}
		\end{equation*} Here the first equality is again due to fact that the contribution from $\omega_{\geq 2}$ is zero, and the second equality follows from the fact that $g'=1$ on $|x_1|<3R_1$ and $|p-\tau_1(t)|\leq 2R_1$. Note that the last integral involving $\tld{\omg}_{1}$ can again be bounded by $C \varep_{1}$ , using $\nrm{\tld{\omg}_{1}}_{L^2} \lesssim \varep_{1}$, with $C>0$ depending on the size of the support of $\tld{\omg}_1$. Therefore we arrive at
		\begin{equation}\label{eq:rd-p-H-bound}
			|\partial_p H(t,p)+m_1|\leq C\varep_1 \quad\text{ for all }\quad|p-\tau_1(t)|\leq 2R_1.
		\end{equation}
		In particular, by choosing $\varep_1$ sufficiently small, \eqref{eq:rd-p-H-bound} implies 
		\begin{equation}
			\label{dHdp2}    \partial_p H(t,p)\leq -\frac{1}{2}m_1 \quad  \text{ for all } \quad |p-\tau_1(t)|\leq 2R_1.
		\end{equation}

		Combining \eqref{eq:H-bound-tau} and \eqref{dHdp2}, for any $t\in [t_0, \min\{t_0+T_1, T\}]$, there exists a unique $p_{1}(t)$ in $(\tau_1(t)-2R_1, \tau_1(t)+2R_1)$ satisfying \eqref{def_p}, and it also satisfies \eqref{temp_p_diff}.
		In particular, at $t=0$, this gives $|p(0)-\bar p_0|\leq C\varep$ since $\tau_1(0)$ can be chosen as $\bar p$.
		
		Putting \eqref{temp_p_diff} together with the triangle inequality and Lipschitz continuity of $\omg_{Lamb}^{\bar\kpp_{1}, \bar\mu_{1}}$, we obtain \eqref{temp_goal_p}:
		\begin{equation*}
			\begin{split}
				\nrm{ \omg_{1}(t) - \omg_{Lamb}^{\bar\kpp_{1}, \bar\mu_{1}}(\cdot - p_{1}(t) \mathbf{e}_{1} ) }_{L^{2} \cap  L^1_*  } &\le \nrm{ \omg_{1}(t) - \omg_{Lamb}^{\bar\kpp_{1}, \bar\mu_{1}}(\cdot - \tau_{1}(t) \mathbf{e}_{1} ) }_{L^{2} \cap  L^1_*  } \\
				&\qquad + \nrm{ \omg_{Lamb}^{\bar\kpp_{1}, \bar\mu_{1}}(\cdot)  - \omg_{Lamb}^{\bar\kpp_{1}, \bar\mu_{1}}(\cdot - (\tau_{1}(t) - p_{1}(t)) \mathbf{e}_{1} ) }_{L^{2} \cap  L^1_*  }  \\
				&\le C\varep_{1} + C|\tau_{1}(t) - p_{1}(t)| \le 2C\varep_{1}. 
			\end{split}
		\end{equation*} 
		
		\medskip
		\noindent\underline{Step 4. Estimating derivative of $p_1(t)$.} To justify the induction step (b), we proceed similarly as in \cite[\S 3.2]{JYZ}, which we briefly sketch below. (The only difference is that \cite[\S 3.2]{JYZ} considers the equation in the moving frame with velocity $\bar V_1$, whereas we consider it in the original frame.)
		
		Recall that $p_1(t)$ is defined to satisfy $H(t,p_1(t))=0$. Since $H(t,p)$ is differentiable in $t,p$, and $\partial_p H(t,p_1(t))$ does not vanish by \eqref{dHdp2}, we have $p_1(t)$ is differentiable  for $t\in [t_0, \min\{T_1,T\}]$, with its derivative given by
		\begin{equation}\label{dpdt}
			\frac{d}{dt}p_1(t) = -\frac{\partial_t H(t,p_1(t))}{\partial_p H(t,p_1(t))}.
		\end{equation}
		Note that we already have the estimate \eqref{eq:rd-p-H-bound} for the denominator. For the numerator, we have
		\[
		\partial_t H(t,p_1(t)) = \int_{ \bbR_+^2 } u_1(t, x+p_1(t)\mathbf{e}_1) \,\omg_1(t, x+p_1(t)\mathbf{e}_1) \,g'(x_1) dx.
		\]
		In \eqref{temp_goal_p}, we have shown $\omg_1(t, x+p_1(t)\mathbf{e}_1)$ is close to $\omg_{Lamb}^{\bar \kappa_1,\bar \mu_1}$. In the above integral, if we approximate $\omg_1(t, x+p_1(t)\mathbf{e}_1)$ and $u_1(t, x+p_1(t)\mathbf{e}_1)$ by $\omg_{Lamb}^{\bar \kappa_1,\bar \mu_1}$ and $(u_{Lamb}^{\bar \kappa_1,\bar \mu_1})_1$ respectively, the integral would become
		\[
		\begin{split}
			I_{approx}:=\int_{ \bbR_+^2 } (u_{Lamb}^{\bar \kappa_1,\bar \mu_1})_1(x) \,\omg_{Lamb}^{\bar \kappa_1,\bar \mu_1}(x)\, g'(x_1) dx 
			= \bar V_1\int_{ \bbR_+^2 }  \omg_{Lamb}^{\bar \kappa_1,\bar \mu_1}(x) g'(x_1) dx
			= \bar V_1 m_1.
		\end{split}
		\]
		Here the first equality is due to the fact that $(u_{Lamb}^{\bar\kappa_1, \bar\mu_1} - \bar V_1 \mathbf{e}_1)\cdot\nabla \omg_{Lamb}^{\bar\kappa_1, \bar\mu_1}=0$ (recall $\omg_{Lamb}^{\bar\kappa_1, \bar\mu_1}$ is a traveling wave solution with velocity $\bar V_1 \mathbf{e}_1$); testing this equation by $g(x_1)$ and applying the divergence theorem gives the first equality. The second inequality follows from the definition $m_1:=\int_{ \bbR_+^2 } \omg_{Lamb}^{\bar \kappa_1,\bar \mu_1} dx$, and the fact that $g'=1$ in $|x_1|<3R_1$. 
		
		It remains to estimate the difference $\partial_t H(t,p_1(t)) - I_{approx}$. By \eqref{temp_goal_p}, $\|\omg_1(t, x+p_1(t)\mathbf{e}_1)-\omg_{Lamb}^{\bar \kappa_1,\bar \mu_1}\|_{L^2}\leq C\varep_1$. Using the Biot--Savart law gives $\|u(t, x+p_1(t)\mathbf{e}_1)-u_{Lamb}^{\bar \kappa_1,\bar \mu_1}\|_{L^\infty}\leq C\varep_1^{1/2}$. These estimates (together with the fact that $\omega_1$ has compact support) gives
		\[
		|\partial_t H(t,p_1(t)) - \bar V_1 m_1| = |\partial_t H(t,p_1(t)) - I_{approx}|\leq C \varep_1^{1/2}.
		\]
		Plugging this and \eqref{eq:rd-p-H-bound} into \eqref{dpdt} gives \eqref{eq:p1-B-goal}.
		
		\medskip 
		
		Finally, given $\varep$ sufficiently small, we may take $\varep_{1} = \frac{\varep}{10CN}$. This finishes the proof. 
	\end{proof}

	\subsection{Velocity on the borders} 
	As a consequence of \eqref{eq:p1-B}, we obtain smallness of the velocity everywhere between the two ``borders'' $L_{i}^{\pm}(t)$.   In particular, since such estimate also holds for $x_1=L_i(t)$, the smallness of velocity (as compared to $\dot L_i$) implies that it is only possible for a fluid particle to cross a ``border'' $L_i(t)$ from the right to the left, but never the other way. 
	\begin{lemma}\label{lem:u-border}
		If the bootstrap assumptions \eqref{eq:B1}--\eqref{eq:B3} hold for $[0,T]$, then for all $t\in [0,T]$ and all $1 \le i \le N-1$, we have \begin{equation}\label{eq:u-border}
			\begin{split}
				\sup_{ \{x_1\in [L_i^-(t), L_i^+(t)], x_{2} \ge 0\} }|u(t,x)| \le \varep^{1/2}\dot{L}_{i}.
			\end{split}
		\end{equation}  
	\end{lemma}
	
	\begin{proof} 
		From Lemma \ref{lem:approx-Lamb-dipoles}, we have that \begin{equation}\label{eq:vort-approx}
			\begin{split}
				\Big\|\omg(t,\cdot) - \sum_{j=1}^{N} \bar{\omg}_{j,Lamb}(t,\cdot)\Big\|_{L^{2}\cap L^1_* } \le \varep. 
			\end{split}
		\end{equation} We decompose the velocity $u =: u_{Lamb} + u_{rem}$, where $u_{Lamb}$ denotes the velocity generated by $\sum_{j=1}^{N} \bar{\omg}_{j,Lamb}(t,\cdot)$; here $\bar{\omg}_{j,Lamb}(t,\cdot)$ is defined by \eqref{eq:omg-1-Lamb-epsilon}. When $x = (x_1,x_2)$ satisfies $x_{1} \in [L_{i}^-(t), L_i^+(t)]$ for some $i$, the support of $\sum_{j=1}^{N} \bar{\omg}_{j,Lamb}(t,\cdot)$ is separated from $x$ by at least $D_{0}/10$ by \eqref{eq:p1-B}. Applying \eqref{eq:u-decay}, \begin{equation*}
			\begin{split}
				|u_{Lamb}(x)| \le \frac{C}{D_{0}^{2}}.
			\end{split}
		\end{equation*} Next, applying \eqref{eq:vel-L-infty-alpha} with $\alp=1/2$ and \eqref{eq:vort-approx}, we obtain \begin{equation*}
			\begin{split}
				|u_{rem}(x)| \le C(M^{1/3}+\varep^{1/3})\varep^{2/3} \le C\varep^{2/3} 
			\end{split}
		\end{equation*} where we let $C$ also depend on $M$. Taking $\varep_{0}$ smaller, if necessary, relative to all $\dot L_i$ ($1 \le i \le N-1$) and $M$ and then taking $D_{0}$ larger in a way depending on $\varep_{0}$, we obtain  that \eqref{eq:u-border} holds for all $x_{1} \in [L_{i}^-(t), L_i^+(t)]$.
	\end{proof}

	\subsection{Bounds on the gain} \label{subsec:gain-bounds}
	
	\begin{lemma}\label{lem:gain-bounds}
		If the bootstrap assumptions \eqref{eq:B1}--\eqref{eq:B3} hold for $[0,T]$, then for all $t\in [0,T]$ and all $2 \le i \le N$, we have the following bounds for $\omg_{\ge i, gain}$: 
		\begin{equation}\label{eq:gain-bounds}
			\begin{split}
				{\nrm{\omg_{\ge i,gain}(t,\cdot)}_{L^{2}} +  \nrm{\omg_{\ge i,gain}(t,\cdot)}_{L^{1}} } \le C  \dlt^{1/2}.
			\end{split}
		\end{equation} 
	\end{lemma}
	\begin{proof} 
		Recall that in {Lemma \ref{lem:u-border}}, we have shown that fluid particles can only cross a ``border'' $L_i(t)$ from right to left, but not the other way. Since $\omega$ is conserved along the particle trajectory, $\omega_{\geq i, gain}(t,\cdot)$ is purely contributed by the fluid particles who have crossed $L_i$ from the right to the left before time $t$ (i.e. the fluid particles who left $\sum_{j<i}\omega_j$ by time $t$). Using this fact together with \eqref{eq:B1},
		we immediately have the following bound on $\nrm{\omg_{\ge i,gain}}_{L^{2}}^{2}$: \begin{equation*}
			\begin{split}
				\nrm{\omg_{\ge i,gain}(t,\cdot)}_{L^{2}}^{2} { \le  \sum_{j = 1}^{N-1} } \left( K_{j}(0) - K_{j}(t) \right)  \le C  \dlt . 
			\end{split}
		\end{equation*} 
		The $L^1$ bound then follows from the Cauchy--Schwarz inequality and the fact that $\omega_0$ has a compact support, whose measure is conserved in time:
		\[
		\int_{\mathbb{R}^2_+} \omg_{\ge i,gain}(t,x) dx \leq \nrm{\omg_{\ge i,gain}(t,\cdot)}_{L^{2}} |\supp\omega_0|^{1/2} \leq C\delta^{1/2}. \qedhere 
		\]
	\end{proof}

	\subsection{Bounds on the errors} We recall from \eqref{eq:omg-err-def}--\eqref{eq:omg-cen-def} that we have defined \begin{equation*}
		\begin{split}
			\omg_{i,err}^{(\ell)} = \omg_{i}\,\mathbf{1}_{ \{  L_{i}(t) \le x_{1} < L_{i}^{+}(t) \} }, \,\, \omg_{i,cen} = \omg_{i}\,\mathbf{1}_{ \{  L_{i}^{+}(t) \le x_{1} < L_{i-1}^{-}(t) \}}, \,\, \omg_{i,err}^{(r)} = \omg_{i}\,\mathbf{1}_{ \{  L_{i-1}^{-}(t) \le x_{1} < L_{i-1}(t) \} } 
		\end{split}
	\end{equation*}  for all $1 \le i \le N$. We have $\omg^{(r)}_{1,err} = 0 = \omg^{(\ell)}_{N,err}$. We denote the corresponding stream functions by $\psi_{i,err}^{(\ell)}, \psi_{i,cen}$, and $\psi_{i,err}^{(r)}$. 
	
	\begin{lemma}[$\dlt$-smallness of errors]\label{lem:omg-err-small}
		If the bootstrap assumptions \eqref{eq:B1}--\eqref{eq:B3} hold for $[0,T]$, then for all $t\in [0,T]$, we have \begin{equation}\label{eq:omg-err-small}
			\begin{split}
				\nrm{ \omg_{i,err}^{(\ell)} }_{L^1_* } + \nrm{\omg_{i,err}^{(\ell)} }_{L^{2}}^{2} + \nrm{ \omg_{i,err}^{(r)} }_{L^1_* } + \nrm{\omg_{i,err}^{(r)} }_{L^{2}}^{2} \le C\dlt. 
			\end{split}
		\end{equation}
	\end{lemma}

	\begin{proof}
		Let us denote $\omega_{i,err} := \omega_{i,err}^{(\ell)}+\omega_{i,err}^{(r)}$, and denote by $\psi_{i,err}$ its corresponding stream function. So \eqref{eq:omg-err-small} is equivalent with
		\[
		\nrm{ \omg_{i,err} }_{L^1_* } + \nrm{\omg_{i,err} }_{L^{2}}^{2}\le C\dlt. 
		\]
		For simplicity, we set $\omg_{i,rem}(t) := \omg_{i}(t) - \bar\omg_{i,Lamb}(t)$, which satisfies \eqref{eq:omg-1-Lamb-epsilon}. (Note that $\omg_{i,rem}$ contains the left and right errors $\omg_{i,err}$, as well as the error in the central part.)
		For any $x\in\supp \omega_{i,err}$, we have $x_1 \in (L_i(t), L_i^+(t))\cup (L_{i-1}^-, L_{i+1}(t))$, therefore
		\begin{equation*}
			\begin{split}
				|\nb \psi_{i}(t,x) | \le |\nb \psi_{i,rem}(t,x)| + |\nb\psi_{i,Lamb}(t,x)| < \frac{C_{L}}{10} \bar\kpp_{i} \quad\text{ for } x\in\supp \omega_{i,err}
			\end{split}
		\end{equation*} using that $\nrm{\omg_{i,rem}}_{L^2}<\varep<\varep_{0}$ by \eqref{eq:omg-1-Lamb-epsilon}, and that $\bar\omg_{i,Lamb}(t)$ is far from $L_{i}^{+}(t)$ and $L_{i-1}^{-}(t)$ by \eqref{eq:p1-B}. 
		
		Similarly, using $\|\omega_{i,err}\|_{L^2}\leq {\|\omega_{i,rem}\|_{L^2}} \leq \varep_0$, we also have
		$
		|\nb \psi_{i,err}(t,x) | <  \frac{C_{L}}{10} \bar\kpp_{i},
		$
		and subtracting the above two inequalities gives
		$
		|\nb \psi_{i,cen}(t,x) | <  \frac{C_{L}}{5} \bar\kpp_{i}$ for $x\in\supp \omega_{i,err}$. Since all these stream functions are zero on the $x_1$-axis, we have
		\begin{equation}
			\label{temp_psi}
			\psi_{i,err}(t,x) < \frac{C_{L}}{10} \bar\kpp_{i}x_2 \quad\text{ and }\psi_{i,cen}(t,x) < \frac{C_{L}}{5} \bar\kpp_{i}x_2 \quad \text{ for }x\in\supp \omega_{i,err}.
		\end{equation}
		
		We now decompose and estimate $E_{i}(t)$ as follows: writing $\mu_{i,cen} = \nrm{\omg_{i,cen}}_{L^1_* }, \kpp_{i,cen} = \nrm{\omg_{i,cen}}_{L^{2}}$, we use the above estimates to obtain \begin{equation}\label{eq:Ei-decomp}
			\begin{split}
				E_{i}(t) & = -\frac12\int \omg_{i,err}\left( \psi_{i,err}+ 2 {\psi}_{i,cen} \right) dx - \frac12\int {\omg}_{i,cen} {\psi}_{i,cen} dx  \\
				& \le   \frac1{4}C_{L}\bar\kpp_{i}\nrm{\omg_{i,err}(t)}_{L^1_* }  + C_{L} {\mu}_{i,cen}(t) {\kpp}_{i,cen}(t) \\
				& \le  \frac1{4}C_{L}\bar\kpp_{i}\nrm{\omg_{i,err}(t)}_{L^1_* }  + C_{L} (\bar\mu_{i} + C\dlt - \nrm{\omg_{i,err}(t)}_{L^1_* } ) (\bar{\kpp}_{i} + C\dlt) \\
				& \le C_{L} \bar\kpp_{i}\bar\mu_{i} - \frac34 C_{L}\bar{\kpp}_{i}  \nrm{\omg_{i,err}(t)}_{L^1_* } +   C\dlt. 
			\end{split}
		\end{equation}
		Now combining this with (B3)$_{i}$ gives \begin{equation*}
			\begin{split}
				C_{L} \bar\kpp_{i}\bar\mu_{i}  - C\dlt < E_{i}(0) - {C} \dlt < E_{i}(t) \le C_{L} \bar\kpp_{i}\bar\mu_{i} - \frac23 C_{L}\bar{\kpp}_{i}  \nrm{\omg^{(\ell)}_{i,err}(t)}_{L^1_* } +   C\dlt 
			\end{split}
		\end{equation*} This gives $\nrm{\omg_{i,err}(t)}_{L^1_* } \le C\dlt$. Plugging this back into the second inequality of \eqref{eq:Ei-decomp} and combining it with (B3)$_{i}$ again, we have
		\[
		C_{L} \bar\kpp_{i}\bar\mu_{i}  - C\dlt < E_i(t) \leq C_L \bar\mu_i \kpp_{i,cen}(t) + C\delta.
		\]
		This yields $\kpp_{i,cen}(t)\geq \bar\kpp_i - C\delta$, leading to
		$\nrm{\omg_{i,err}(t)}_{L^{2}}^{2} \le C\dlt$.   
	\end{proof}

	\subsection{Bounds on the interaction energy}\label{subsec:E-inter-small}
	
	\begin{lemma}\label{lem:E-inter-small}
		If the bootstrap assumptions \eqref{eq:B1}--\eqref{eq:B3} hold for $[0,T]$, then for all $t\in [0,T]$ we have  \begin{equation}\label{eq:E-inter-small}
			\begin{split}
				0 \le E[\omg] - \sum_{1 \le i \le N } E[\omg_{i}] \le C\dlt^{3/2} . 
			\end{split}
		\end{equation}
	\end{lemma}
	
	\begin{proof}
		
		We take $D_{0} > 1/\dlt$ and decompose \begin{equation*}
			\begin{split}
				E[\omg] & = \sum_{1 \le i,j \le N} E_{inter}[\omg_{i},\omg_{j}]  = \sum_{1 \le i \le N } E[\omg_{i}] + 2\sum_{1 \le i \le N-1 }E_{inter}[\omg_{i},\omg_{i+1}] + \sum_{i,j : |i-j|>1} E_{inter}[\omg_{i},\omg_{j}]. 
			\end{split}
		\end{equation*} Each term on the right-hand side is nonnegative. Then, we may bound using \eqref{eq:Energy-interaction-support-separation} \begin{equation*}
			\begin{split}
				\sum_{i,j : |i-j|>1} E_{inter}[\omg_{i},\omg_{j}] \le \frac{C}{D_{0}^{2}} \le C\dlt^{2}. 
			\end{split}
		\end{equation*} Similarly, for $\sum_{1 \le i \le N-1 }E_{inter}[\omg_{i},\omg_{i+1}] $, we decompose $\omg_{i} = (\omg_{i} - \omg_{i,err}^{(\ell)}) + \omg_{i,err}^{(\ell)}$ and $\omg_{i+1} = (\omg_{i+1} - \omg_{i+1,err}^{(r)}) + \omg_{i+1,err}^{(r)}$ so that \begin{equation*}
			\begin{split}
				E_{inter}[\omg_{i},\omg_{i+1}] =  E_{inter}[\omg_{i} - \omg_{i,err}^{(\ell)},\omg_{i+1}] + E_{inter}[\omg_{i,err}^{(\ell)}, \omg_{i+1} - \omg_{i+1,err}^{(r)}] + E_{inter}[\omg_{i,err}^{(\ell)} ,\omg_{i+1,err}^{(r)} ]
			\end{split}
		\end{equation*} and bound again using \eqref{eq:Energy-interaction-support-separation} \begin{equation*}
			\begin{split}
				\sum_{1 \le i \le N-1} \left( E_{inter}[\omg_{i} - \omg_{i,err}^{(\ell)},\omg_{i+1}] + E_{inter}[\omg_{i,err}^{(\ell)}, \omg_{i+1} - \omg_{i+1,err}^{(r)}] \right) \le \frac{100C}{D_{0}^{2}} \le C\dlt^{2}. 
			\end{split}
		\end{equation*} Finally, we use \eqref{eq:Energy-interaction-X} to bound \begin{equation*}
			\begin{split}
				E_{inter}[\omg_{i,err}^{(\ell)} ,\omg_{i+1,err}^{(r)} ] \le C\dlt^{3/2},
			\end{split}
		\end{equation*} where we have applied \eqref{eq:omg-err-small} from Lemma \ref{lem:omg-err-small}. 
	\end{proof}  
	
	\section{Closing the Bootstrap Assumptions}\label{sec:boot}
	
	In this section, we keep working under the assumptions \eqref{eq:B1}--\eqref{eq:B3} on $[0,T]$ for some $T>0$ and prove that actually all the assumptions \eqref{eq:B1}--\eqref{eq:B3} hold on $[0,T]$ with $\dlt$ replaced by $\dlt/100$, using all of the lemmas established in \S \ref{sec:boot-con}. By continuity in time, this immediately concludes that \eqref{eq:B1}--\eqref{eq:B3} hold with $\dlt/100$ for all $t \in [0,\infty)$. In particular, Lemma \ref{lem:approx-Lamb-dipoles} is applicable for all nonnegative times, which establishes Theorem \ref{thm:N-Lamb-dipoles}. 
	
	\subsection{Enstrophy, impulse, and energy change}\label{subsec:formulae}
	
	Given a time-dependent function $g(t,x)$  that is transported by the velocity field $u$ on {$\mathbb{R}^2_+$}, it is convenient to introduce the notation $\mathbf{F}_{i}[g](t)$ to denote the flux of $g$ across the moving line $\{x_1=L_i(t)\}$ (with the leftward direction defined as positive flux): \begin{equation*}
		\begin{split}
			\mathbf{F}_{i}[g](t) := \int_{ \{ x_{1} = L_{i}(t), x_2\geq 0 \} } \left( \dot{L}_{i} - u_{1}(t,x) \right) g(t,x) dx_{2}. 
		\end{split}
	\end{equation*} This expression is interpreted as $0$ for $i = 0$ and $N$; $L_{0}=+\infty$ and $L_{N}=-\infty$ were introduced just for convenience. 
	
	\begin{proposition}
		We have the following formulae for the enstrophy, impulse, and energy change: \begin{equation}\label{eq:Ki-dt}
			\begin{split}
				\dot{K}_{i}(t) & = \mathbf{F}_{i-1}[\omg^{2}](t) - \mathbf{F}_{i}[\omg^{2}](t), 
			\end{split}
		\end{equation}\begin{equation}\label{eq:Ki-dt-sum}
			\begin{split}
				\dot{K}_{\le j}(t) & = - \mathbf{F}_{j}[\omg^{2}](t), 
			\end{split}
		\end{equation} \begin{equation}\label{eq:mui-dt}
			\begin{split}
				\dot{\mu}_{i}(t) & = \sum_{j \ne i} \int  u_{j,2}(t,x)\omg_{i}(t,x) dx + \mathbf{F}_{i-1}[x_{2}\omg](t) - \mathbf{F}_{i}[x_{2}\omg](t),
			\end{split}
		\end{equation} \begin{equation}\label{eq:mui-dt-sum}
			\begin{split}
				\dot{\mu}_{\le j}(t) & =\int   (u_{\ge j+1})_{2}(t,x)   \,  \omg_{\le j}(t,x)  dx - \mathbf{F}_{j}[x_{2}\omg](t),
			\end{split}
		\end{equation}  and \begin{equation}\label{eq:Ei-dt} 
			\begin{split}
				\dot{E}_{i}(t) = \sum_{j \ne i}  \int  -\nb\psi_{i}(t,x) \cdot (u_{j}\omg_{i})(t,x) \,  dx + \mathbf{F}_{i-1}[-\psi_{i}\omg](t) - \mathbf{F}_{i}[-\psi_{i}\omg](t). 
			\end{split}
		\end{equation}  
	\end{proposition}
	The rest of this section is devoted to the proof of this proposition.
	
	\subsubsection{Enstrophy} We compute 
	\begin{equation*}
		\begin{split}
			\dot K_{i}(t) & = \frac{d}{dt} \int_{ \{L_{i}(t) < x_{1} < L_{i-1}(t)\} } \omg^{2}(t,x) dx \\ 
			& = \int_{ \{ x_{1} = L_{i-1}(t) \} } \left( \dot{L}_{i-1} - u_{1}(t,x) \right) \omg^{2}(t,x) dx_{2} -  \int_{ \{ x_{1} = L_{i}(t) \} } \left( \dot{L}_{i} - u_{1}(t,x) \right) \omg^{2}(t,x) dx_{2} .
		\end{split}
	\end{equation*} This gives \eqref{eq:Ki-dt}. Summation in $i$ gives \eqref{eq:Ki-dt-sum}. 
	
	\subsubsection{Impulse}
	We compute 
	\begin{equation*}
		\begin{split}
			\dot  \mu_{i}(t) & = \frac{d}{dt} \int_{ \{L_{i}(t) < x_{1} < L_{i-1}(t)\} } x_{2}\omg(t,x) dx \\ 
			& = \int_{ \{L_{i}(t) < x_{1} < L_{i-1}(t)\} } u_{2}(t,x)\omg(t,x) dx \\
			& \qquad + \int_{ \{ x_{1} = L_{i-1}(t) \} } \left( \dot{L}_{i-1} - u_{1}(t,x) \right) x_{2}\omg(t,x) dx_{2} -  \int_{ \{ x_{1} = L_{i}(t) \} } \left( \dot{L}_{i} - u_{1}(t,x) \right) x_{2}\omg(t,x) dx_{2} \\
			& = \sum_{j \ne i} \int_{ \{L_{i}(t) < x_{1} < L_{i-1}(t)\} } u_{j,2}(t,x)\omg_{i}(t,x) dx \\
			& \qquad + \int_{ \{ x_{1} = L_{i-1}(t) \} } \left( \dot{L}_{i-1} - u_{1}(t,x) \right) x_{2}\omg(t,x) dx_{2} -  \int_{ \{ x_{1} = L_{i}(t) \} } \left( \dot{L}_{i} - u_{1}(t,x) \right) x_{2}\omg(t,x) dx_{2} 
		\end{split}
	\end{equation*} where we used that \begin{equation*}
		\begin{split}
			\int_{ \{L_{i}(t) < x_{1} < L_{i-1}(t)\} } u_{i,2}(t,x)\omg_{i}(t,x) dx = 0 
		\end{split}
	\end{equation*} by anti-symmetry of the Biot--Savart kernel. Therefore, we obtain \eqref{eq:mui-dt}.  Summation in $i$ gives \eqref{eq:mui-dt-sum}, with the observation \begin{equation*}
		\begin{split}
			\int  u_{j,2}(t,x)\omg_{i}(t,x) dx = \int (\rd_{x_1}\lap^{-1} \omg_{j})(t,x)\omg_{i}(t,x) =  -  \int  u_{i,2}(t,x)\omg_{j}(t,x) dx . 
		\end{split}
	\end{equation*}
	
	\subsubsection{Energy} 
	Recall that the energy $E_i(t)$ is given by
	\[
	E_i(t) = E[\omega_i(t)] =  \frac{1}{2}\int_{ \{L_{i}(t) < x_{1} < L_{i-1}(t)\} } \int_{ \{L_{i}(t) < y_{1} < L_{i-1}(t)\} } \omega(t,x)\omega(t,y)\mathbf{G}(x,y)dxdy.
	\]
	When taking its time derivative, due to the symmetry between $x$ and $y$, we can apply the derivative only to $\omega(t,x)$ and the integral limits of $x$, and remove the coefficient $1/2$. This yields the following, where we also use the definition $\psi_i(t,x) := -\int_{ \{L_{i}(t) < y_{1} < L_{i-1}(t)\} } \omega(t,y)\mathbf{G}(x,y)dy$:
	\begin{equation}\label{temp_E}
		\begin{split}
			\dot{E}_i(t) &= \int_{ \{L_{i}(t) < x_{1} < L_{i-1}(t)\} } -\psi_{i}(t,x) (\rd_{t}\omg)(t,x) dx + \\
			&\qquad + \dot{L}_{i}(t) \int_{ \{ x_{1} = L_{i}(t) \} } \psi_{i}(t,x) \omg(t,x) dx_{2} -  \dot{L}_{i-1}(t) \int_{ \{ x_{1} = L_{i-1}(t) \} } \psi_{i}(t,x) \omg(t,x) dx_{2} . \\
		\end{split}
	\end{equation} For the first integral on the right-hand side, we have \begin{equation*}
		\begin{split}
			&\int_{ \{L_{i}(t) < x_{1} < L_{i-1}(t)\} } -\psi_{i}(t,x) (\rd_{t}\omg)(t,x) dx = \int_{ \{L_{i}(t) < x_{1} < L_{i-1}(t)\} } \psi_{i}(t,x) \nb\cdot((u\omg)(t,x)) \,  dx \\
			& \quad =  \int_{ \{L_{i}(t) < x_{1} < L_{i-1}(t)\} } -\nb\psi_{i}(t,x) \cdot (u \omg_{i})(t,x) \,  dx \\
			&\qquad - \int_{ \{ x_{1} = L_{i}(t) \} } \psi_{i}(t,x)(u_{1}\omg)(t,x)dx + \int_{ \{ x_{1} = L_{i-1}(t) \} }\psi_{i}(t,x)(u_{1}\omg)(t,x)dx 
		\end{split}
	\end{equation*} and we further note that, \begin{equation*}
		\begin{split}
			\int_{ \{L_{i}(t) < x_{1} < L_{i-1}(t)\} } -\nb\psi_{i}(t,x) \cdot (u\omg_{i})(t,x) \,  dx  = \sum_{j \ne i}  \int_{ \{L_{i}(t) < x_{1} < L_{i-1}(t)\} } -\nb\psi_{i}(t,x) \cdot (u_{j}\omg_{i})(t,x) \,  dx 
		\end{split}
	\end{equation*} holds, using a similar cancellation to the impulse case.

	Plugging the above into \eqref{temp_E} yields \begin{equation*}
		\begin{split}
			&	\dot  E_{i}(t)  = \sum_{j \ne i}  \int_{ \{L_{i}(t) < x_{1} < L_{i-1}(t)\} } -\nb\psi_{i}(t,x) \cdot (u_{j}\omg_{i})(t,x) \,  dx \\
			&\qquad +  \int_{ \{ x_{1} = L_{i-1}(t) \} }  (\dot{L}_{i-1}(t) - u(t,x))(-\psi_{i}\omg)(t,x) dx_{2}  -  \int_{ \{ x_{1} = L_{i}(t) \} }  (\dot{L}_{i}(t) - u(t,x))(-\psi_{i}\omg)(t,x) dx_{2}  .
		\end{split}
	\end{equation*} This gives \eqref{eq:Ei-dt}.  
	\subsection{Closing the assumption on enstrophy}
	
	\begin{proposition}\label{prop:enstrophy-upper-bound}
		We have for all $1 \le i \le N-1$ that  \begin{equation}\label{eq:enstrophy-upper-bound}
			\begin{split}
				\sup_{t \in [0,T]}   K_{\le i}(t) \le  K_{\le i}(0) .
			\end{split}
		\end{equation}
	\end{proposition}
	
	\begin{proof}
		Using \eqref{eq:u-border}, we have that $\bfF_{i}[\omg^{2}](t) \ge 0$ for any $1 \le i \le N-1$. Therefore,   \begin{equation*}
			\begin{split}
				\dot K_{\le i}(t) = - \bfF_{i}[\omg^{2}](t) \le 0, 
			\end{split}
		\end{equation*} for any $1 \le i \le N-1$. Integrating in $t \in [0,T]$ gives \eqref{eq:enstrophy-upper-bound}. 
	\end{proof}
	
	\subsection{Closing the assumption on impulse}\label{subsec:impulse}
	
	The following proposition is at the heart of the proof, where we aim to control the impulse change for all times. The proof starts with an explicit expression of the time derivative of impulse, where we decompose the vorticity left to border into the ``gained'' part and the original part, and decompose the vorticity right to border into the ``error'' that is close to border, and the centered part that is further away. 
	
	Out of these combinations, the most difficult case is the interaction between the ``gained'' part on the left and the ``error'' part on the right, since their support are not separated apart -- both parts can be arbitrarily close to the border. Note that the smallness of these two terms themselves is not sufficient: we need to show the \emph{time integral} (in an arbitrarily long time interval) contributed by the interaction is small. This key estimate is done in {Lemma~\ref{lem:gain-error}}, where we switch to a Lagrangian viewpoint to control the time integral of the contribution from a ``gained'' particle. 
	\begin{proposition}\label{prop:impulse-upper-bound} We have for all $1 \le i \le N-1$ that 
		\begin{equation}\label{eq:impulse-upper-bound}
			\begin{split}
				\sup_{t \in [0,T]} \mu_{\le i}(t) \le \mu_{\le i}(0) + \frac{\dlt}{100}.
			\end{split}
		\end{equation} 
	\end{proposition}
	\begin{proof}	We consider the case $i = 1$, and begin by recalling \eqref{eq:mui-dt}: \begin{equation}\label{eq:def-I}
			\begin{split}
				\dot \mu_{1}(t) +  \mathbf{F}_{1}[x_{2}\omg](t) & = \sum_{j>1} \int  (u_{j})_{2}(t,x)\omg_{1}(t,x) dx =: I(t). 
			\end{split}
		\end{equation}  The  term $ \mathbf{F}_{1}[x_{2}\omg](t)$ is nonnegative by Lemma \ref{lem:u-border}. {The right-hand side (denoted by $I(t)$) can be rewritten as follows, with \begin{equation*}
				\begin{split}
					(u_{1})_{2}(t,x) = \int_{ \{ y_{1} \ge L_{1}(t) \} } K_{2}(x,y) \omg_{1}(t,y) \, dy , \qquad K_{2}(x,y) :=  -\frac{1}{2\pi} \frac{x_1 - y_1}{|x-y|^{2}} \, \frac{4x_2y_2}{|x-\bar{y}|^{2}}
				\end{split}
			\end{equation*} denoting the second component of $u_1$: \begin{equation*}
				\begin{split}
					I(t) & := \int_{ \bbR_+^2 } \left( \int_{\{ y_{1} < L_{1}(t) \}} K_{2}(x,y) \omg(t,y) \, dy  \right)  \omg_{1}(t,x) dx \\
					& = - \int_{\{ y_{1} < L_{1}(t) \}}  \left( \int_{ \{ x_{1} \ge L_{1}(t) \} } K_{2}(y,x) \omg_{1}(t,x) dx \right)  \omg(t,y) dy \\
					& 
					= \int_{ \{ x_{1} < L_{1}(t) \} } \omg(t,x) (u_{1})_{2} (t,x) dx, 
				\end{split}
			\end{equation*} where we have switched from $y$ and $x$ used the cancellations \begin{equation*}
				\begin{split}
					\int_{ \bbR_+^2 } \int_{ \bbR_+^2 } K_{2}(x,y) \omg(t,x)\omg(t,y) \, dx dy = 0  = \int_{ \bbR_+^2 } \int_{ \bbR_+^2 } K_{2}(x,y) \omg_{1}(t,x)\omg_{1}(t,y) \, dx dy
				\end{split}
			\end{equation*} (by antisymmetry in $x$ and $y$) in the last equality.} We recall the decomposition \eqref{eq:omg-gain}--\eqref{eq:omg-gain-def}: \begin{equation*}
			\begin{split}
				\omg(t,x) \, \mathbf{1}_{ \{ x_{1} < L_{1}(t)  \} } = \omg_{\ge2}(t,x) = \omg_{\ge2,*}(t,x) + {\omg}_{\ge 2,gain}(t,x),
			\end{split}
		\end{equation*} where ${\omg}_{\ge 2, gain}(t,\cdot)$ is the part of $\omg_{\ge 2}$ which have crossed the line $L_{1}(t)$ from right to left during the time interval $[0,t]$. Correspondingly, we decompose $I(t) = I_*(t) + I_{gain}(t)$, where \begin{equation*}
			\begin{split}
				I_{*}(t) &=  \int_{ \{ x_{1} < L_{1}(t) \} } \omg_{\ge2, *}(t,x) (u_{1})_{2}(t,x) \, dx \\
				& = -\frac1{2\pi} \int_{ \{ x_{1} < L_{1}(t) \} }  \int_{ \{ y_{1} \ge L_{1}(t) \} }  \frac{x_1 - y_1}{|x-y|^{2}} \, \frac{4x_2y_2}{|x-\bar{y}|^{2}}  \,  \omg_{\ge2, *}(t,x)  \omg_{1}(t,y)     \, dy dx . 
			\end{split}
		\end{equation*} We further decompose $I_*(t) = I_{*, 1}(t) + I_{*, 2}(t) + I_{*, 3}(t)$, where \begin{equation*}
			\begin{split}
				I_{*, 1}(t) = -\frac1{2\pi} \int_{ \{ x_{1} < L_{1}(t) \} }  \int_{ \{ y_{1} \ge L_{1}^{+}(t) \} }  \frac{x_1 - y_1}{|x-y|^{2}} \, \frac{4x_2y_2}{|x-\bar{y}|^{2}}  \,  \omg_{\ge2, *}(t,x)  \omg_{1}(t,y)     \, dy dx , 
			\end{split}
		\end{equation*}   \begin{equation*}
			\begin{split}
				I_{*, 2}(t) = -\frac1{2\pi} \int_{ \{ x_{1} < L_{1}^{-}(t) \} }  \int_{ \{ L_{1}(t) \le y_{1} < L_{1}^{+}(t) \} }  \frac{x_1 - y_1}{|x-y|^{2}} \, \frac{4x_2y_2}{|x-\bar{y}|^{2}}  \,  \omg_{\ge2, *}(t,x)  \omg_{1}(t,y)     \, dy dx , 
			\end{split}
		\end{equation*}
		\begin{equation*}
			\begin{split}
				I_{*, 3}(t) = -\frac1{2\pi} \int_{ \{ L_{1}(t) \ge x_{1} \ge L_{1}^{-}(t) \} }  \int_{ \{ L_{1}(t) \le y_{1} < L_{1}^{+}(t) \} }  \frac{x_1 - y_1}{|x-y|^{2}} \, \frac{4x_2y_2}{|x-\bar{y}|^{2}}  \,  \omg_{\ge2, *}(t,x)  \omg_{1}(t,y)     \, dy dx.  
			\end{split}
		\end{equation*} We can estimate $I_{*, 1}(t)$ and $I_{*, 2}(t)$ similarly: note that on their integral domains we have \begin{equation*}
			\begin{split}
				|x-y|, |x-\bar{y}| \ge |x_1 -y_1| \ge \frac{1}{10}\left(D_{0} +  (\bar{V}_{1} - \bar{V}_{2} )t \right) 
			\end{split}
		\end{equation*} and we can therefore bound the absolute value of the kernel by \begin{equation*}
			\begin{split}
				\frac{C x_{2} y_{2}}{\left(D_{0} +  (\bar{V}_{1} - \bar{V}_{2} )t \right)^{3} }. 
			\end{split}
		\end{equation*} Using simply that the impulse of $\omg_{1}$ and $\omg_{\ge 2, *}$ is bounded by a time-independent constant,  we have \begin{equation}\label{eq:I-star}
			\begin{split}
				\int_0^T\left|I_{*,1}(t) \right| + \left|I_{*,2}(t) \right|  dt \le C\int_0^{T}  \frac{1}{\left(D_{0} +  (\bar{V}_{1} - \bar{V}_{2} )t \right)^{3} } \, dt \le C\dlt^{2} \le \dlt^{3/2}, 
			\end{split}
		\end{equation} by taking $D_{0}>1/\dlt$ and $\dlt$ smaller if necessary. Note that the constants are independent of $T$. 
		
		In the case of $I_{*,3}(t)$, we observe that by Lemma~\ref{lem:u-border} and the definition of $\omg_{\ge 2, *}$, the distance between $\supp( \omg_{\ge 2, *}(t,\cdot) ) \cap \{x_1 \in [L_1^-(t), L_1(t)]\}$ and the border $\{x_1=L_1(t)\}$ is at least $(1-\varep^{1/2})\dot L_1 t \geq \frac{1}{4}(\bar V_1 + \bar V_2) t$. Therefore, for all $x \in \supp( \omg_{\ge 2, *}(t,\cdot) ) \cap \{x_1 \in [L_1^-(t), L_1(t)]\}$ and $y \in \supp( \omg_{1}(t,\cdot) )$, we have  \begin{equation*}
			\begin{split}
				|x-y|, |x-\bar{y}| \ge |x_1 -y_1| \ge L_1(t)-x_1 \ge \frac14 (\bar V_1 + \bar V_2) t.
			\end{split}
		\end{equation*} We use this bound just for $t \ge T_{\dlt}$, for some $T_{\dlt}>0$ to be determined: proceeding similarly as before, we get this time \begin{equation*}
			\begin{split}
				\int_{T_{\dlt}}^T\left|I_{*,3}(t) \right|  dt \le C\dlt^{2}\int_{T_\dlt}^{T}  \frac{1}{\left( (\bar{V}_{1} + \bar{V}_{2} )t \right)^{3} } \, dt \le \frac{C \dlt^{2}}{T_{\dlt}^{2}},
			\end{split}
		\end{equation*} where we have used that \begin{equation*}
			\begin{split}
				\int_{ \{ L_{1}(t) \ge x_{1} \ge L_{1}^{-}(t) \} } x_{2}\omg_{\ge 2, *}(t,x) dx \le C\dlt , \quad   \int_{ \{ L_{1}(t) \le y_{1} < L_{1}^{+}(t) \} } y_{2} \omg_{1}(t,y) dy \le C\dlt. 
			\end{split}
		\end{equation*} To treat $0<t\le T_{\dlt}$, we just do the brute force bound  
		\begin{equation}
			\begin{split}
				|I_{*, 3}(t)| &\le  C\int_{ \{ L_{1}(t) \ge x_{1} \ge L_{1}^{-}(t) \} }  \int_{ \{ L_{1}(t) \le y_{1} < L_{1}^{+}(t) \} }  \frac{1}{|x-y|}  \,  \omg_{\ge2, *}(t,x)  \omg_{1}(t,y)     \, dy dx \\
				& \le C \dlt^{1/4} \nrm{\omg_{1}(t,y)   \mathbf{1}_{ \{ L_{1}(t) \le y_{1} < L_{1}^{+}(t) \} }}_{L^1} \le \dlt^{3/4}, 
			\end{split}
		\end{equation}  where we used \eqref{eq:vel-L-infty-alpha} with $\alp = 1/2$ for $\omg = \omg_{\ge 2, *}(t,x) \mathbf{1}_{\{ L_{1}(t) \ge x_{1} \ge L_{1}^{-}(t) \}}$ which has $L^{2}$ norm bounded by $C\dlt$. Hence, combining two bounds,  \begin{equation}\label{I3}
			\begin{split} \int_{0}^T\left|I_{*,3}(t) \right|  dt  \le \frac{C \dlt^{2}}{T_{\dlt}^{2}} +C T_{\dlt} \dlt^{3/4} \le C\dlt^{7/6},
			\end{split}
		\end{equation} by selecting $T_{\dlt} = \dlt^{5/12}$. 
		
		\medskip 
		
		To deal with $I_{gain}(t)$, 
		we now use the decomposition\footnote{For $i \ge 2$, we need to also take into account $u_{i,err}^{(r)}$, but it can be handled together with $u_{i,cen}$.}  
		\begin{equation*}
			\begin{split}
				u_{1}(t,x) = u_{1,err}^{(\ell)}(t,x) + u_{1,cen}(t,x) 
			\end{split}
		\end{equation*} where we recall that \begin{equation*}
			\begin{split}
				u_{1,err}^{(\ell)}(t,x) = -\frac1{2\pi} \int_{ \{ L_{1}(t) \le y_{1} < L_{1}^{+}(t) \} } \frac{x_1 - y_1}{|x-y|^{2}} \, \frac{4x_2y_2}{|x-\bar{y}|^{2}}  \omg(t,y)       \, dy. 
			\end{split}
		\end{equation*} 
		Then similarly as in the above, we can derive \begin{equation}\label{eq:I-gain}
			\begin{split}
				\int_0^{T}  \left| \int_{ \{ x_{1} < L_{1}(t) \} }  {\omg}_{\ge2,gain}(t,x) (u_{1,cen})_{2}(t,x) \, dx\right| \, dt \le \dlt^{2}
			\end{split}
		\end{equation}  using separation of support: if $x \in \supp(\omg_{\ge2,gain})$ and $L_{1}^{+}(t) < y_{1}$ then \begin{equation*}
			\begin{split}
				|x-y| \ge \frac1{10}\left( D_{0} + (\bar{V}_{1} - \bar{V}_{2})t \right). 
			\end{split}
		\end{equation*}
		We now claim the following key estimate. 
		\begin{lemma}[Gain--Error interaction]\label{lem:gain-error} We have 
			\begin{equation}\label{eq:key-Lagrangian}
				\begin{split}
					\int_0^{T}  \int_{ \{ x_{1} < L_{1}(t) \} } {\omg}_{\ge2, gain}(t,x) (u_{1,err}^{(\ell)})_{2}(t,x) \, dx  \, dt  \le \dlt^{4/3}. 
				\end{split}
			\end{equation} 
		\end{lemma}
		Assuming \eqref{eq:key-Lagrangian}, the proof of \eqref{eq:impulse-upper-bound} for $i = 1$ is complete, by combining it with \eqref{eq:I-star}, \eqref{I3} and \eqref{eq:I-gain}, and taking $\dlt$ smaller than $10^{-10}$, if necessary. The proof for $1 < i \le N - 1$ is completely parallel with this case, using the decomposition $\omg_{\ge i+1} = \omg_{\ge i+1, *} + \omg_{\ge i+1, gain}$. 
	\end{proof}
	
	\begin{proof}[Proof of Lemma \ref{lem:gain-error}]
		To begin with, we write $F := (u_{1,err}^{(\ell)})_{2}$ and make an area preserving change of variables $x \mapsto \Phi(t,x)$ to rewrite the integral as \begin{equation*}
			\begin{split}
				\int_0^{T} \int_{ \{ x_{1} \ge L_{1}(0) \} \cap \{ \Phi_{1}(t,x) < L_{1}(t) \} } \omg_{0}(x) \, F(t,\Phi(t,x)) \, dx \, dt ,
			\end{split}
		\end{equation*}  where we use the definition of ${\omg}_{\ge2, gain}(t,\cdot)$: recall that every fluid particle in its support originates on the right of $L_1(0)$ at time $0$, and has crossed to the left of $L_1(t)$ by time $t$. Using Fubini's theorem, the integral equals \begin{equation*}
			\begin{split}
				\int_{ \{ x_{1} \ge L_{1}(0) \} } \omg_{0}(x) \, \mathbf{1}_{ \{ \tau(x) \le T \} }(x) \left[  \int_{ \min\{\tau(x),T\}}^{T} F(t,\Phi(t,x)) \,dt  \right]   dx, \quad \tau(x) := \inf\{ t \ge 0:  \Phi_{1}(t,x) < L_{1}(t) \}. 
			\end{split}
		\end{equation*} We set $\tau(x) = +\infty$ if $\Phi_{1}(t,x) \ge L_{1}(t)$ for all $t\ge0$. 
		
		We have a uniform in time bound for $F$, applying \eqref{eq:vel-L-infty-alpha} with $\alp = 3/4$ together with \eqref{eq:omg-err-small}:  \begin{equation}\label{eq:F1}
			\begin{split}
				\nrm{F(t,\cdot)}_{L^{\infty}} \le C \dlt^{3/8} .  
			\end{split}
		\end{equation}

		Furthermore, we observe that once a ``fluid particle'' starting initially at $x$ crosses the line $\{ x_{1} = L_{1}(t) \}$ from the right to left at time $t=\tau(x)$ (i.e. $\Phi_{1}(\tau(x),x) = L_{1}(\tau(x))$, by Lemma~\ref{lem:u-border}, its distance to the borderline $L_{1}(t)$ increases linearly:
		we have that \begin{equation*}
			\begin{split}
				|L_{1}(t) - \Phi_{1}(t,x) | = L_{1}(t) - \Phi_{1}(t,x)  \ge \frac1{10}|\bar{V}_{1} - \bar{V}_{2}|(t - \tau(x)) \qquad \text{ for all } t \ge \tau(x).
			\end{split}
		\end{equation*}  Hence, for $T > t > \tau(x)$, we have a pointwise estimate from \eqref{eq:u-decay}: \begin{equation}\label{eq:F2}
			\begin{split}
				F(t,\Phi(t,x)) \le C\nrm{ \omg_{1,err}^{(\ell)} }_{L^1_* } |t - \tau(x)|^{-2} \le C\dlt |t - \tau(x)|^{-2} , \qquad t \in (\tau(x), T). 
			\end{split}
		\end{equation}  
		
		Equipped with two bounds \eqref{eq:F1} and \eqref{eq:F2}, we can now integrate for each $x$: \begin{equation*}
			\begin{split}
				\left|\int_{\min\{ \tau(x), T \} }^{T} F(t,\Phi(t,x)) \,dt\right| \le C\int_{\min\{ \tau(x), T \} }^{T}  \min\left\{ \dlt^{3/8} , \dlt|t-\tau(x)|^{-2}  \right\} dt \le C\dlt^{11/16}. 
			\end{split}
		\end{equation*} Finally, we recall the $L^{1}$ bound \eqref{eq:gain-bounds} from Lemma \ref{lem:gain-bounds}: \begin{equation*}
			\begin{split}
				\int_{ \{ x \ge L_{1}(0) \} } \omg_{0}(x) \, \mathbf{1}_{ \{ \tau(x) < T \} }(x) dx =  \nrm{\omg_{\ge2, gain}(T,\cdot)}_{L^{1}} \le C \dlt^{ 1/2 }. 
			\end{split}
		\end{equation*} Therefore, \begin{equation*}
			\begin{split}
				\int_{ \{ x \ge L_{1}(0) \} } \omg_{0}(x) \, \mathbf{1}_{ \{ \tau(x) < T \} }(x) \,  \left[\int_{\min\{ \tau(x), T \} }^{T} F(t,\Phi(t,x)) \,dt \right] \, dx \le C\dlt^{5/8+1/2} \le \dlt^{9/8}, 
			\end{split}
		\end{equation*} which finishes the proof. 
	\end{proof}
	
	\subsection{Closing the assumption on energy}\label{subsec:energy}
	\begin{proposition}\label{prop:energy-upper-bound} For all $1 \le i \le N$, we have 
		\begin{equation}\label{eq:energy-upper-bound}
			\begin{split}
				\inf_{t \in [0,T]} E_{i}(t) \ge E_{i}(0) - \frac{\dlt}{100}.
			\end{split}
		\end{equation} 
	\end{proposition}
	\begin{proof}
		We recall the formula \eqref{eq:Ei-dt}: \begin{equation*}
			\begin{split}
				\dot E_{i}(t) = \sum_{j \ne i}  \int  -\nb\psi_{i}(t,x) \cdot (u_{j}\omg_{i})(t,x) \,  dx + \mathbf{F}_{i-1}[-\psi_{i}\omg](t) - \mathbf{F}_{i}[-\psi_{i}\omg](t). 
			\end{split}
		\end{equation*}  
		
		We first take $1 \le i \le N-1$, and let us prove a stronger statement, namely \begin{equation}\label{eq:E-N}
			\begin{split}
				\sup_{t \in [0,T]} \left| E_{i}(t) - E_{i}(0) \right| \le \frac{\dlt}{1000N}, \qquad 1 \le i \le N-1. 
			\end{split}
		\end{equation} Note that from \eqref{eq:params-dlt-close} and \eqref{eq:impulse-upper-bound}, we actually have for all $1 \le i \le N-1$ that \begin{equation*}
			\begin{split}
				0 \le \int_0^{T} \bfF_{i}[x_{2}\omg](t) dt \le C\dlt. 
			\end{split}
		\end{equation*} Now the key observation is that \begin{equation*}
			\begin{split}
				\left| \mathbf{F}_{i}[-\psi_{i}\omg](t)\right| \le \sup_{ x_{1} = L_{i}(t) }\left| \frac{\psi_{i}(t,x)}{x_{2}} \right| \bfF_{i}[x_{2}\omg](t) \le C\varep^{1/2} \bfF_{i}[x_{2}\omg](t) ,
			\end{split}
		\end{equation*} where the smallness estimate on $\frac{\psi_{i}(t,x)}{x_{2}}$ follows along the same lines as the proof of Lemma \ref{lem:u-border}. In particular, by taking $\varep_{0}$ smaller, we can guarantee that \begin{equation*}
			\begin{split}
				\int_0^T 	\left| \mathbf{F}_{i}[-\psi_{i}\omg](t)\right| dt \le \frac{\dlt}{10000N}, \qquad 1 \le i \le N-1.
			\end{split}
		\end{equation*} Therefore, it suffices to obtain a similar bound for \begin{equation}\label{eq:E-ij}
			\begin{split}
				E_{i,j}(T) := \int_{0}^{T} \int  -\nb\psi_{i}(t,x) \cdot (u_{j}\omg_{i})(t,x) \,  dx dt , \qquad i \ne j, 1 \le i \le N-1. 
			\end{split}
		\end{equation} Bounding this in absolute value by $o(\dlt)$ is somewhat simpler than the proof of Proposition \ref{prop:impulse-upper-bound}, the factor $-\nb\psi_{i}$ giving an additional smallness compared to the corresponding integral for the impulse change. {We provide some details of the proof as a separate proposition; see Proposition \ref{prop:E-ij} below.}
		
		Finally, when $i = N$, we use \eqref{eq:E-N}, conservation of energy $E[\omg(t,\cdot)] = E[\omg_{0}]$, together with \eqref{eq:E-inter-small} and initial closeness \eqref{eq:ini-smallness} to obtain \eqref{eq:energy-upper-bound} for $i = N$, by taking $C_{0}$ larger if necessary ($\dlt_{0} = \dlt/C_{0}$). This finishes the proof. 
	\end{proof}

	\begin{proposition}\label{prop:E-ij} 
		For each $i, j$ satisfying $i \ne j, 1 \le i \le N-1$, we have \begin{equation}\label{eq:E-ij-bound}
			\begin{split}
				E_{i,j}(T) \le C\left(D_{0}^{-1} + \dlt^{9/8}\right). 
			\end{split}
		\end{equation}
	\end{proposition}
	\begin{proof}
		We only consider the case $i = 1, j = 2$, since the other cases are similar. To ease notation, we also denote $u_{j} = v$, and denote its two components by $v_1$ and $v_2$. Then we note that $E_{1,2}(T)$ consists of two terms: \begin{equation*}
			\begin{split}
				E_{1,2}(T) = -\int_{0}^{T} \int  \rd_{1}\psi_{1}(t,x)   (v_{1}\omg_{1})(t,x) \,  dx dt  - \int_{0}^{T} \int  \rd_{2}\psi_{1}(t,x)   (v_{2}\omg_{1})(t,x) \,  dx dt  . 
			\end{split}
		\end{equation*} We further observe that the second term is bounded in absolute value by \begin{equation*}
			\begin{split}
				\nrm{ \rd_{2}\psi_{1} }_{L^{\infty}([0,T]\times\bbR^2_+)}\int_{0}^{T} \int   (v_{2}\omg_{1})(t,x) \,  dx dt 
			\end{split}
		\end{equation*} since $v_{2}\omg_{1} \ge 0$ pointwise in time and space. Since $\nrm{ \rd_{2}\psi_{1} }_{L^{\infty}([0,T]\times\bbR^2_+)}  \le C$ for some $C$ independent of $T$ and the integral is bounded by $I(t)$ defined in \eqref{eq:def-I}, we conclude  that \begin{equation*}
			\begin{split}
				\left|  - \int_{0}^{T} \int  \rd_{2}\psi_{1}(t,x)   (v_{2}\omg_{1})(t,x) \,  dx dt  \right| \lesssim \dlt^{9/8}. 
			\end{split}
		\end{equation*} To handle the other term, which we denote by $II$ from now on, we first write out $v_{1}(t,x)$ as an integral against $\omg_{2}$ using \eqref{eq_u1}: \begin{equation*}
			\begin{split}
				II = -\int_{0}^{T} \int_{ \{ y_{1} \le L_{1}(t) \} }  \int_{ \{ x_{1} > L_{1}(t) \} } \frac{({(x_1-y_1)^2 + }(x_2^2-y_2^2)) y_2\, \omg_{2}(t,y) }{\pi |x-y|^2 |x-\bar{y}|^2} \rd_{1}\psi_{1}(t,x)    \omg_{1}(t,x) \,  dx dy  dt .
			\end{split}
		\end{equation*} Taking absolute values and using $|x_2^2-y_2^2| \le |x-y||x-\bar{y}|$, we have the bound \begin{equation}
			\begin{split}\label{II_estimate}
				|II| \le \int_{0}^{T} \int_{ \{ y_{1} \le L_{1}(t) \} } \int_{ \{ x_{1} > L_{1}(t) \} } \frac{y_{2}\omg_{2}(t,y) }{\pi |x-y| |x-\bar{y}|} |\rd_{1}\psi_{1}(t,x)|    \omg_{1}(t,x) \,  dx dy  dt .
			\end{split}
		\end{equation} 
		
		To bound the integral on the right hand side, we split its integral domains of $y, x$ into sub-domains $\{y_1 \leq L_1^-(t)\} \cup \{y_1 \in [ L_1^-(t), L_1(t)]\}$ and $ \{x_1 \in ( L(t), \tilde{L}_1^+(t)]\} \cup \{x_1 \geq \tilde{L}_1^+(t)\}$ respectively. Here $\tilde{L}_1^+ := \frac{1}{2}(L_1(t)+L_1^+(t))$, and the reason to introduce it will be made clear soon. 
		
		With such split, the right hand side of \eqref{II_estimate} can be split into a sum of four integrals. Three of them have the integration domains of $x$, $y$ separated from each other by a distance of order $D_0 + t$. These three integrals can be handled similarly as in \eqref{eq:I-star} in the proof of Proposition \ref{prop:impulse-upper-bound}, using the uniform-in-time boundedness of $\nrm{y_{2}\omg_{2}(t,y) }_{L^1}$, $\nrm{\rd_{1}\psi_{1}(t,x)}_{L^\infty}$, $\nrm{ \omg_{1}(t,x)}_{L^1}$ and decay of $|x-y| |x-\bar{y}|$. For instance, for the integral in $\{x_{1} > L_{1}^{+}(t), \, y_{1} \in [L_1^-(t), L_{1}(t)]\}$,  we have $|x-\bar{y}| \ge |x-y| \gtrsim D_{0} + t$, which gives the bound \begin{equation*}
			\begin{split}
				&\int_{0}^{T} \int_{ \{ y_{1} \in [L_1^-(t), L_{1}(t)] \} } \int_{ \{ x_{1} > \tilde{L}_{1}^{+}(t) \} } \frac{y_{2}\omg_{2}(t,y) }{\pi |x-y| |x-\bar{y}|} |\rd_{1}\psi_{1}(t,x)|    \omg_{1}(t,x) \,  dx dy  dt \\
				&\qquad \le C\nrm{y_{2}\omg_{2}(t,y) }_{L^1} \nrm{\rd_{1}\psi_{1}(t,x)}_{L^\infty} \nrm{ \omg_{1}(t,x)}_{L^1} \int_{0}^{T} \frac{1}{(D_{0}+t)^{2}} dt \le C D_{0}^{-1}. 
			\end{split}
		\end{equation*} 
		However, the remaining integral in the region $\{y_1 \in [L_1^-(t), L_1(t)], \,x_1\in (L_1(t), \tilde{L}_1^+(t)]\}$ is more delicate to control, since $x_1-y_1$ can be arbitrarily small. To treat this integral \begin{equation*}
			\begin{split}
				II_{err}:=\int_{0}^{T} \int_{ \{ L_{1}^{-}(t) <  y_{1} \le L_{1}(t) \} } \int_{ \{ L_{1}(t) < x_{1} \le \tilde{L}_{1}^{+}(t) \} } \frac{y_{2}\omg_{2}(t,y) }{\pi |x-y| |x-\bar{y}|} |\rd_{1}\psi_{1}(t,x)|    \omg_{1}(t,x) \,  dx dy  dt, 
			\end{split}
		\end{equation*} we first prove that \begin{equation*}
			\begin{split}
				|\rd_{1}\psi_{1}(t,x)| \leq |u_1(t,x)| \le C(\dlt^{7/16} + D_0^{-1}) \quad\text{ for all }t\in [0,T], x\in [L_1(t), \tilde{L}_1^+(t)].
			\end{split}
		\end{equation*} 
		This is done by decomposing $u_1$ into the contribution from $\omega_{1,err}^{(l)}$ and $\omega_{1,cen}$. The estimate \eqref{eq:omg-err-small} on $\omega_{1,err}^{(l)}$ allows us to apply \eqref{eq:vel-L-infty-alpha} to bound its contribution to $u_1$ by  $C\dlt^{7/16}$ (the power can be made arbitrarily close to $1/2$, but $7/16$ is sufficient for us).  The contribution from $\omega_{1,cen}$ can be easily controlled by $C D_0^{-1}$, since there is an order $D_0$ distance between $\tilde{L}_1^+$ and $L_1^+$; this is the motivation for introducing $\tilde{L}_1^+$.
		
		Next, we follow a similar idea as in the bound of the ``gain--error interaction'' Lemma~\ref{lem:gain-error} to rewrite $II_{err}$ using a Lagrangian viewpoint. However, we now follow the trajectory of the vorticity in $\{x_1 \in (L_1(t), \tilde L_1^+(t)]\}$, instead of trajectory of ``gain particles'' in $y_1 \in [L_1^-(t), L_1(t)]$.\footnote{If we try to control $II_{err}$ in the same way as Lemma~\ref{lem:gain-error},  the spatial integral will only decay like $1/t$, which is not integrable in time if $T\gg 1$.} Namely, through the change of variables $x = \Phi(t,z)$, we rewrite $II_{err}$ as \begin{equation*}
			\begin{split}
				II_{err} & \le C(\dlt^{7/16} + D_0^{-1})\int_{0}^{T} \int_{ \{ L_{1}(t) < \Phi_{1}(t,z) \le \tilde{L}_{1}^{+}(t) \} } \underbrace{\int_{ \{ L_{1}^{-}(t) <  y_{1} \le L_{1}(t) \} }  \frac{y_{2}\omg_{2}(t,y) }{|\Phi(t,z)-y| |\Phi(t,z)-\bar{y}|}   \,  dy}_{=: G(t,z)} \omg_{0}(z)  dz  dt.
			\end{split}
		\end{equation*} 
		We then define
		\begin{equation*}
			\begin{split}
				A(T) := \left\{ z \in \supp(\omg_{0}) : \, \mbox{there exists } \tau \in [0,T] \mbox{ such that } \Phi_{1}(\tau,z) \in [L_{1}(\tau), \tilde{L}_{1}^{+}(\tau)]  \right\},
			\end{split}
		\end{equation*} and for $z \in A(T)$, we define two time moments $\tau_{in}$ and $\tau_{out}$ by \begin{equation*}
			\begin{split}
				\tau_{in}(z) := \inf \left\{ \tau \in [0,T] :   \Phi_{1}(\tau,z) \in [L_{1}(\tau), \tilde{L}_{1}^{+}(\tau)] \right\}, \quad \tau_{out}(z) := \sup \left\{ \tau \in [0,T]  : \Phi_{1}(\tau,z) \in [L_{1}(\tau), \tilde{L}_{1}^{+}(\tau)] \right\}.
			\end{split}
		\end{equation*} 
		Note that for $z\in A(T)$, we have $ \Phi_1(t,z) \in [L_1(t), \tilde L_1^+(t)]$ in the whole time interval $[\tau_{in}(z), \tau_{out}(z)]$: this is because fluid particles can only cross the borders from right to left, due to Lemma~\ref{lem:u-border}.
		
		With the above definition, we have
		\[
		II_{err} \leq  C(\dlt^{7/16} + D_0^{-1}) \int_{ \bbR_+^2 } \left[ \int_{\tau_{in}(z)}^{  \tau_{out}(z) } G(t,z)  dt \right]  \omg_{0}(z)\mathbf{1}_{A(T)}(z)    dz .
		\]
		To control this, on the one hand, we have the uniform-in-time bound for $G(t,z)$: for any $0<r<R$, \begin{equation*}
			\begin{split}
				G(t,z) &\le \int_{ \mathbb{R}^2_+ }  \frac{ \omg_{2,err}^{(r)}(t,y) }{|\Phi(t,z)-y|  } \, dy \lesssim \|\omg_{2,err}^{(r)}\|_{L^1} r + \|\omg_{2,err}^{(r)}\|_{L^2} (\log\frac{R}{r})^{1/2} + \|\omg_{2,err}^{(r)}\|_{L^\infty} R^{-1} \\
				&\leq  C\dlt^{1/2}|\log\delta|^{1/2} \leq C\dlt^{7/16},
			\end{split}
		\end{equation*}
		where the third inequality is obtained by setting $r=1$, $R=\delta^{-1/2}$, and applying the bounds $\|\omg_{2,err}^{(r)}\|_{L^2}<\delta^{1/2}$ and $\|\omg_{2,err}^{(r)}\|_{L^1}\leq \delta^{1/2}$ (the latter one follows from applying Cauchy--Schwarz and using the boundedness of support).
		
		On the other hand, for any $z\in A(T)$ and $t\in [\tau_{in}(z), \tau_{out}(z)]$, we have $\Phi_1(t,z) - L_1(t) \gtrsim \tau_{out}(z)-t$ from the linear-in-time separation by Lemma~\ref{lem:u-border}\footnote{Lemma~\ref{lem:u-border} says that any Lagrangian particle in the region $[L_1(t),L_1^+(t)]$ has velocity $O(\varepsilon^{1/2})$. Since $L_1(t)$ moves to the right with constant speed, the horizontal distance between the particle to the vertical line $\{x_1=L_1(t)\}$ has to decrease linearly in time, until the moment $\tau_{out}$ when it hits the line. As a result, if we go backwards in time from time $t_{out}(z)$, the particle needs to separate from $\{x_1=L_1(t)\}$ with linear separation.}. As a result, for any $y\in [L_1^-(t), L_1(t)]$, we have
		$|\Phi(t,z)-y| > |\Phi_1(t,z) - L_1(t) |  \gtrsim \tau_{out}(z)-t$. Combining this with Lemma~\ref{lem:omg-err-small}, we have the following time-dependent bound for all $z\in A(T), t\in [\tau_{in}(z), \tau_{out}(z)]$: \begin{equation*}
			\begin{split}
				G(t,z) &\le C \int_{ \{ L_{1}^{-}(t) <  y_{1} \le L_{1}(t) \} }  \frac{y_{2}\omg_{2}(t,y) }{ |\tau_{out}(z) - t |^{2} } \, dy \le \frac{C\|\omega_{2,err}^{(r)}\|_{L^1_*}}{|\tau_{out}(z) - t |^{2}} \leq  \frac{C\dlt}{|\tau_{out}(z) - t |^{2}} .
			\end{split}
		\end{equation*} 
		These two bounds allow us to estimate\begin{equation*}
			\begin{split}
				\int_{\tau_{in}(z)}^{  \tau_{out}(z) } G(t,z)  dt \le C \int_{\tau_{in}(z) - \tau_{out}(z)}^{ 0 } \min\left\{ \dlt^{7/16}, \frac{\dlt}{t^2} \right\} \, dt \le C \dlt^{23/32}. 
			\end{split}
		\end{equation*} This gives \begin{equation*}
			\begin{split}
				II_{err} \le C( \delta^{7/16}+ D_0^{-1}) \dlt^{23/32 } \int_{ \bbR_+^2 }  \omg_{0}(z) \mathbf{1}_{A(T)}(z)    dz \le C \dlt^{  7/16 + 23/32 } + D_0^{-1}. 
			\end{split}
		\end{equation*} 
		This finishes the proof, since $7/16 + 23/32 > 9/8$.
	\end{proof} 
	
	\begin{proof}[Proof of Theorem \ref{thm:N-Lamb-dipoles}]
		Since \eqref{eq:B1}--\eqref{eq:B3} hold on $[0,T]$ with $\dlt$ replaced by $\dlt/100$, by continuity in time, this concludes that \eqref{eq:B1}--\eqref{eq:B3} hold with $\dlt/100$ for all $t \in [0,\infty)$. In particular, Lemma \ref{lem:approx-Lamb-dipoles} is applicable for all nonnegative times, which establishes Theorem \ref{thm:N-Lamb-dipoles}. The proof is complete.
	\end{proof} 
	
	\section{Separating out a Lamb dipole}\label{sec:B} 
	
	In this section, we sketch the proof of Theorem \ref{thm:B}, which is largely parallel to the one of Theorem \ref{thm:N-Lamb-dipoles}. For convenience, let us recall the notation and assumptions from the statement. The initial data $\omg_{0} \ge 0$ on $\bbR^2_+$ is decomposed into $\omg_{0} = \omg_{0 l} + \omg_{0 r}$: \begin{itemize}
		\item $\omega_{0r}$ satisfies $\nrm{ \omg_{0 r} - \bar\omega_{Lamb}(\cdot-D\mathbf{e}_1)}_{L^2 \cap L^1_*  } \leq \delta_0$ and $\|\omega_{0r}\|_{L^1\cap L^\infty} \le M$,
		\item $\omega_{0l}$ satisfies $\supp(\omega_{0l}) \subset\{x_1 < x_2 \cot\alpha\}$ and $V_{max}[\omega_{0l}] < \sin\alpha$.
	\end{itemize}
	We have defined  \begin{equation*}
		\begin{split}
			V_{max}[\omg_{0 l}] = \sup_{g \in \calR[\omg_{0 l}]} \nrm{ \nb^\perp\lap^{-1}(g) }_{L^{\infty}}, \quad V_{avr} = \frac{1}{2} \left(1+\frac{ V_{max}[\omega_{0l}]  }{\sin\alpha} \right)<1, 
		\end{split}
	\end{equation*} where $\calR[\omg_{0 l}]$ is the set of all possible rearrangements of $\omg_{0 l}$ \eqref{eq:rearrange}.
	
	\subsection{Decomposition and bootstrap assumptions}
	
	Let $\omg(t,\cdot)$ be the solution corresponding to $\omg_{0}$, and let us decompose \begin{equation}\label{eq:decompose-B}
		\begin{split}
			\omg_{r}(t,\cdot) := \omg(t,\cdot)\mathbf{1}_{\{x_1 > L(t,x_{2}) \}} , \quad L(t,x_{2}) := x_2 \cot\alpha + V_{avr} t ,  \quad \omg_{l}(t,\cdot) := \omg(t,\cdot) - \omg_{r}(t,\cdot) 
		\end{split}
	\end{equation} and set \begin{equation*}
		\begin{split}
			K_{r}(t) = \kpp_{r}^2(t) = \nrm{ \omg_{r}(t,\cdot) }_{L^{2}}^{2}, \quad \mu_{r}(t) = \nrm{ \omg_{r}(t,\cdot) }_{L^1_* }, \quad E_{r}(t) = E[\omg_{r}(t,\cdot)]. 
		\end{split}
	\end{equation*} Similarly, we define $K_{l}, \kpp_{l}, \mu_{l}$, and $E_{l}$.  	We now introduce the bootstrap assumptions: for a sufficiently small $\dlt > 0$ depending only on $\varep, \alp, M$, we take 
	\begin{itemize}
		\item (B1') $\sup_{ \{ x_{1} = L(t,x_{2}) \} }( u_{1}(t,x) - \cot \alp \, u_{2}(t,x) )  \le V_{avr}$. 
		
		\item (B2') $\mu_{r}(t) \le \mu_{r}(0) + \dlt$.

		\item (B3') $E_{r}(t) \ge E_{r}(0) - \dlt$.
		
	\end{itemize}

	Note that (B2') is exactly identical to \eqref{eq:B2} for the $N=2$ case. (B3') is identical to \eqref{eq:B3} (for $i=1$ only). Here we do not need to control $E_l(t)$, since the bootstrap assumption \eqref{eq:B3} for $i=N$ is only needed to show that $\omega_N$ stays close to a Lamb dipole; whereas in Theorem~\ref{thm:B} we do not need to prove anything on $\omega_l$.

	The condition (B1') says that the fluid particles can cross $\{x_{1} = L(t,x_{2})\}$ only from right to left, and compared to the case of Theorem \ref{thm:N-Lamb-dipoles}, this is the only difference in the set of bootstrap assumptions. We shall see that (B1') implies the enstrophy bound $K_{r}(t) \le K_{r}(0)$ and that (B1') is propagated in time by the assumption on $\omg_{0 l}$ involving the rearrangements. Given a sufficiently small $\dlt>0$, we are going to choose $\dlt_{0} = \dlt/C_{0}$ with some large $C_{0}  = C_{0}(\varep,\alp,M)>0$, and prove that the bootstrap assumptions propagate for all $t \ge 0$. 
	
	\subsection{Consequences of the bootstrap assumptions}
	
	We assume that the bootstrap assumptions hold for a time interval $[0,T]$ and obtain their consequences, during the same time interval.
	
	\subsubsection{Upper bound on enstrophy}
	To begin with, we compute the enstrophy change: \begin{equation*}
		\begin{split}
			\dot{K}_{r}(t) & = \frac{d}{dt} \int_{ \{x_1 > L(t,x_{2})\}} \omg^{2}(t,x) dx = \int_{0}^{\infty} \left(\left( u_{1} - \cot\alp u_{2} - V_{avr} \right)\omg^{2}\right)(t, L(t,x_{2}) , x_{2} )  \, dx_{2}
		\end{split}
	\end{equation*}
	and this gives $K_{r}(t) \le K_{r}(0)$ from (B1'). 
	
	\subsubsection{Approximating by Lamb dipole}  Using  $K_{r}(t) \le K_{r}(0)$ with (B2') and (B3'), we may apply Proposition \ref{prop:Lamb-energy-estimate} to the case of normalized Lamb dipole: \begin{equation*}
		\begin{split}
			\inf_{\tau\in\bbR} \nrm{ \omg_{r}(t,\cdot) - \bar{\omg}_{Lamb}(\cdot - \tau \mathbf{e}_{1}) }_{L^2 \cap L^1_* } < \varep .
		\end{split}
	\end{equation*} As in \S \ref{subsec:approx-Lamb-dipoles}, we can show that the shift $\tau(t)$ can be chosen to satisfy  \begin{equation*}
		\begin{split}
			|\tau(t) - t - D| \le C\varep^{1/2}(1+t). 
		\end{split}
	\end{equation*} For convenience, we take $\omg_{r, Lamb}(t,\cdot) =  \bar{\omg}_{Lamb}(\cdot - \tau(t) \mathbf{e}_{1})$ and define $\omg_{r,rem}(t,\cdot) = \omg_{r}(t,\cdot) - \omg_{r,Lamb}(t,\cdot)$.  Furthermore, by combining the upper bounds on $K_{r}(t)$, $\mu_{r}(t)$ with \begin{equation}\label{eq:E-r-UB}
		\begin{split}
			C_{L}\mu_{r}(0)\kpp_{r}(0) - C\dlt < E_{r}(t) \le C_{L}\mu_{r}(t)\kpp_{r}(t) , 
		\end{split}
	\end{equation} we deduce lower bounds \begin{equation}\label{eq:mu-K-r-LB}
		\begin{split}
			\mu_{r}(t) > \mu_{r}(0) - C\dlt , \qquad K_{r}(t) > K_{r}(0) - C\dlt . 
		\end{split}
	\end{equation}

	\subsubsection{Smallness of the gain and error} As in \eqref{eq:omg-gain}--\eqref{eq:omg-gain-def}, we define $\omg_{l, gain}(t,\cdot)$ to be the part of $\omg_{l}(t,\cdot)$ which came from $\omg_{r}$ during the time interval $[0,t]$: \begin{equation}\label{eq:omg-l-gain}
		\begin{split}
			\omg_{l,gain}(t,\cdot) := \omg_{l}(t,\cdot) \mathbf{1}_{ \{\Phi^{-1}_{1}(t,x) \ge L(0, \Phi^{-1}_{2}(t,x) )\} }(\cdot) , \qquad \omg_{l, *}(t,\cdot) := \omg_{l}(t,\cdot) - \omg_{l, gain}(t,\cdot). 
		\end{split}
	\end{equation} Here, $\Phi^{-1}(t,x)=(\Phi^{-1}_1(t,x),\Phi^{-1}_2(t,x))$ is the inverse of $\Phi$, where $\Phi$ is the flow map associated with $u(t,x)$; $\Phi(t,\Phi^{-1}(t,x)) = x = \Phi^{-1}(t,\Phi(t,x))$. 
	
	\begin{lemma}\label{lem:gain-l-bounds} 
		We have \begin{equation*}
			\begin{split}
				\nrm{ \omg_{l, gain}(t,\cdot) }_{L^{2}}^{2} \le C \dlt, \qquad \nrm{ \omg_{l, gain}(t,\cdot) }_{L^{1}} \le C \dlt^{1/2}.
			\end{split}
		\end{equation*} 
	\end{lemma}
	\begin{proof}
		The proof is almost identical to that of Lemma \ref{lem:gain-bounds}, using \eqref{eq:mu-K-r-LB}. 
	\end{proof}

	We then decompose $\omg_{r}$ in a similar way as in the Eulerian decomposition in \S\ref{sec_euler}, into the ``error term'' which is relatively close to the borderline and the rest: \begin{equation}\label{eq:omg-r-err}
		\begin{split}
			\omg_{r,err}(t,\cdot) :=  \omg_{r}(t,\cdot)\mathbf{1}_{\{x_1 \le L^+(t,x_{2}) \}} , \quad L^+(t,x_{2}) := x_2 \cot\alpha + \big(V_{avr} + \frac1{10} \big) t + \frac{D}{2}    
		\end{split}
	\end{equation} and we take $ \omg_{r, cen}(t,\cdot) := \omg_{r}(t,\cdot) - \omg_{r, err}(t,\cdot) $. These two further decomposition of $\omega_l$ and $\omega_r$ are illustrated in Figure~\ref{fig_thmB_decompose}.

	\begin{figure}[htbp]
		\includegraphics[scale=1]{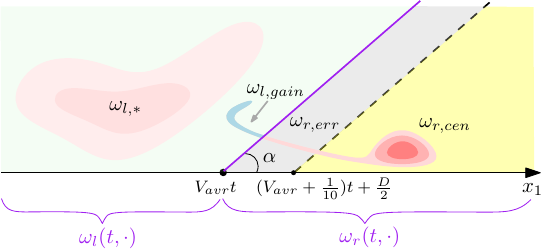} 
		\caption{\label{fig_thmB_decompose} Illustration of the decomposition of $\omega_l$ (vorticity to the left of the slanted purple line) into $\omega_{l,*}$ (red color) and $\omega_{l,gain}$ (blue color), and the decomposition of $\omega_r$ (vorticity to the right of the slanted purple line) into $\omega_{r,err}$ (on the gray background) and $\omega_{r,cen}$ (on the yellow background).} 
	\end{figure}
	
	\begin{lemma}\label{lem:omg-r-err-small}
		We have \begin{equation}\label{eq:omg-r-err-small}
			\begin{split}
				\nrm{  \omg_{r, err}(t,\cdot) }_{L^1_* }  + \nrm{  \omg_{r, err}(t,\cdot) }_{L^{2}}^{2} \le C\dlt.  
			\end{split}
		\end{equation}
	\end{lemma}
	\begin{proof}
		The proof is parallel to that of Lemma \ref{lem:omg-err-small}. 
	\end{proof}

	\subsection{Closing the bootstrap assumptions}
	
	We now close the bootstrap assumptions: on $[0,T]$ where they are valid, we will prove that (B1') holds with the strict inequality, and (B2') and (B3') hold with $\dlt/100$ instead of $\dlt$. This would complete the proof of Theorem \ref{thm:B}. 
	
	\medskip
	\noindent \textbf{Closing (B1')}. 
	We show that on $[0,T]$, we actually have the strict inequality in (B1'). We estimate the velocity at time $t$ on $\{x_{1} = L(t,x_{2})\}$ by the sum of velocities coming from the following decomposition: \begin{equation*}
		\begin{split}
			\omg = \omg_{l, *} + \omg_{l, gain} + \omg_{r, cen} + \omg_{r, err}. 
		\end{split}
	\end{equation*}
	We denote the corresponding velocities by $u = u_{l, *} + u_{l, gain} + u_{r, cen} + u_{r, err}. $ To begin with, smallness of $u_{l, gain}$ and $u_{r, err}$ comes from Lemma \ref{lem:gain-l-bounds} and \ref{lem:omg-r-err-small}, respectively. Next, the smallness of $u_{r, cen}$ in $\{x_{1} = L(t,x_{2})\}$ follows from \eqref{eq:u-decay} in Lemma \ref{lem:vel-decay}, using the fact that the support of $\omg_{r, cen}$ is separated from this line by $c D$ with some $c>0$. Lastly, smallness of $u_{l, *}$ comes from the observation that $\omg_{l, *}(t,\cdot) \in \calR[\omg_{0 l}]$ and in particular $\nrm{ u_{l, *} (t, \cdot) }_{L^\infty} \le V_{max}[\omg_{0 l}] < \sin \alp$. More precisely, for $v(x):= u_{l, *}(t,x)$, we have for every $x$ the inequality \begin{equation*}
		\begin{split}
			|v_{1}(x) - \cot \alp v_{2}(x)| \le \frac1{\sin\alp} |v(x)| \le \frac1{\sin\alp} V_{max}[\omg_{0 l}] < V_{avr}.  
		\end{split}
	\end{equation*} Since the last inequality is strict, we may take $\varep$ smaller and $D$ larger if necessary, to make sure that (B1') holds again with the strict inequality.

	\medskip
	\noindent \textbf{Closing (B2') and (B3')}. 
	The proofs for replacing $\dlt$ by $\dlt/100$ in (B2') and (B3') are parallel to the case of Theorem \ref{thm:N-Lamb-dipoles} for the rightmost dipole. That is, the arguments to close (B2') and (B3') are almost identical to the proofs of case $i=1$ in Proposition~\ref{prop:impulse-upper-bound} and Proposition~\ref{prop:energy-upper-bound}.

	\bibliographystyle{plain}

\end{document}